\documentclass[11pt, a4paper]{amsart}
\usepackage{graphicx, amsfonts, amssymb, color}
\usepackage{amsmath, mathrsfs}

\newtheorem{thm}{Theorem}[section]

\newtheorem{lem}[thm]{Lemma}
\newtheorem{prop}[thm]{Proposition}
\theoremstyle{definition}
\newtheorem{defn}[thm]{Definition}

\theoremstyle{remark}
\newtheorem{rem}[thm]{Remark}

\numberwithin{equation}{section}

\newcommand{\Real}{\mathbb R}

\newcommand{\Integer}{\mathbb Z}
\newcommand{\F}{\mathcal{F}}

\newcommand{\prob}{\mathbb{P}}
\newcommand{\qprob}{\mathbb{Q}}

\newcommand{\expec}{\mathbb{E}}

\newcommand{\cadlag }{c\`adl\`ag}
\newcommand{\tv}{\tilde{v}}

\newcommand{\indic}{\mathbb{I}}
\newcommand{\pare}[1]{\left(#1\right)}
\newcommand{\bra}[1]{\left[#1\right]}

\newcommand{\wt}[1]{\widetilde{#1}}

\newcommand{\nc}{\newcommand}
\nc{\ba}{\begin{array}} \nc{\ea}{\end{array}}
\nc{\be}{\begin{equation}} \nc{\ee}{\end{equation}}
\nc{\bea}{\begin{eqnarray}} \nc{\eea}{\end{eqnarray}}
\nc{\bean}{\begin{eqnarray*}} \nc{\eean}{\end{eqnarray*}}
\nc{\bu}{\bullet} \nc{\nn}{\nonumber} \nc{\cA}{{\mathcal A}}
\nc{\cB}{{\mathcal B}} \nc{\cC}{{\mathcal C}} \nc{\cD}{{\mathcal
D}}  \nc{\cN}{{\mathcal
N}}\nc{\bbD}{\mathbb{D}} \nc{\cG}{{\mathcal G}} \nc{\cF}{{\mathcal
F}}  \nc{\cR}{{\mathcal
R}}\nc{\cU}{{\mathcal U}} \nc{\cH}{{\mathcal H}}
\nc{\cK}{{\mathcal K}} \nc{\cM}{{\mathcal M}} \nc{\cP}{{\mathcal
P}}  \nc{\bbE}{\mathbb{E}} \nc{\bbEQ}{\mathbb{E}^{\mathbb{Q}}}
\nc{\eps}{\varepsilon}\nc{\bbU}{\mathbb{U}}
\nc{\bbEP}{\mathbb{E}_{\mathbb{P}}}\nc{\bbL}{\mathbb{L}}
\nc{\bbP}{\mathbb{P}} \nc{\bbQ}{\mathbb{Q}} \nc{\Om}{\Omega}
\nc{\om}{\omega} \nc{\bbR}{\mathbb{R}} \nc{\bbC}{\mathbb{C}}
\nc{\bfr}{\begin{flushright}} \nc{\efr}{\end{flushright}}
\nc{\dXt}{\Delta X_{t}} \nc{\dXs}{\Delta X_{s}}
\nc{\bs}{\blacksquare} \nc{\dX}{\Delta X} \nc{\dY}{\Delta Y}
\nc{\dnkx}{\left(X(T^{n}_{k})-X(T^{n}_{k-1})\right)}
\nc{\dom}{depth-of-the-market } \nc{\uar}{\uparrow}
\nc{\dar}{\downarrow}\nc{\rar}{\rightarrow} \nc{\esssup}{\mathrm{ess}\mbox{ }\mathrm{sup}}
\nc{\half}{\frac{1}{2}}\nc{\ol}{\overline}
 \nc{\hbE}{\hat{\bbE}}

\nc{\what}{\widehat} \nc{\fhat}{\what{f}}  \nc {\parx}{\frac{\partial}{\partial x}} \nc
{\parw}{\frac{\partial}{\partial w}} \nc
{\parww}{\frac{\partial^2}{\partial w^2}}
\def\rar{\rightarrow}

\def\dar{\downarrow}

\nc{\chf}{\mbox{$\mathbf1$}} \nc{\eid}{\stackrel{d}{=}}

\begin{document}
\title{Point process bridges and weak convergence of insider trading models}
\thanks{This research is supported in part by STICERD at London School of Economics.}
\date{\today}

\author[]{Umut \c{C}etin}
\address{Department of Statistics, London School of Economics and Political Science, 10 Houghton st, London, WC2A 2AE, UK}
\email{u.cetin@lse.ac.uk, h.xing@lse.ac.uk}

\author[]{Hao Xing}

\keywords{point process bridge, Glosten-Milgrom model, Kyle model, insider trading, equilibrium, weak convergence}

\begin{abstract}
 We construct explicitly a bridge process whose distribution, in its own filtration, is the same as the difference of two independent Poisson processes with the same intensity and its time $1$ value satisfies a specific constraint. This construction allows us to show the existence of Glosten-Milgrom equilibrium and its associated optimal trading strategy for the insider. In the equilibrium the insider employs a mixed strategy to randomly submit two types of orders: one type trades in the same direction as noise trades while the other cancels some of the noise trades by submitting opposite orders when noise trades arrive. The construction also allows us to prove that Glosten-Milgrom equilibria converge weakly to Kyle-Back equilibrium, without the additional assumptions imposed in \textit{K. Back and S. Baruch, Econometrica, 72 (2004), pp. 433-465}, when the common intensity of the Poisson processes tends to infinity.
\end{abstract}

\maketitle
\section{Introduction}
In this paper we perform an explicit construction of a particular
bridge process associated to a point process that arises in the
solution of Glosten-Milgrom type insider trading models from Market
Microstructure Theory. Our starting point is the work of Back and
Baruch \cite{Back-Baruch} who studies a class of equilibrium models of
insider trading (of Glosten-Milgrom type) and their convergence to
Kyle model.

In Glosten-Milgrom type insider trading models, there exists an insider
who possesses the knowledge of the time 1 value of the asset given by
the random variable $\tv$. There is
also another class of traders, collectively known as noise traders, who trade without this insider knowledge. Their trades are of the same size and arrive at Poisson times which are assumed to be independent of $\tv$. The insider trades using her extra information in order to maximise her expected wealth at time $1$ but taking into account that her trades move the prices to
her disadvantage since the price is an increasing function of the total
demand for the asset. Moreover, in order to hide her trades, and thus her private information, she will
also submit orders that are of the same size as noise trades. The price of the asset in this market is determined by a market
maker in the equilibrium whose precise definition is given in Section
\ref{sec: model}.

In the specific model that we will study (and also studied in
\cite{Back-Baruch}) $\tv$ takes values in $\{0,1\}$. Since the noise buy and sell orders arrive at Poisson times and are of
the same size, the net $Z$ of cumulative buy and sell noise trades,
after normalization, is given by the difference of two independent Poisson processes. Writing $Y=Z+X$ for the total demand for the asset, where $X$ denotes the trading
strategy of the insider, we will see in Theorem \ref{t:equilibrium} that a Glosten-Milgrom equilibrium exists if
\begin{enumerate}
 \item[(i)] $Y$ in its own filtration has the same distribution as $Z$,
 \item[(ii)] $[Y_1 \geq y] = [\tv=1]$ almost surely for some $y$ to be determined.
\end{enumerate}
The second condition above implies that in the equilibrium the insider drives the process $Y$ so that the event whether $Y_1$ is larger than $y$ is predetermined at time $0$ from the point of view of the insider, since the set $[\tv=1]$ is at the disposal of the insider already at time $0$. Given this characteristic of $Y$, it can be called (with a slight abuse of terminology) a \emph{point process bridge}.

In Section \ref{sec: bridge construction}, we explicitly construct a pure jump process $X$ whose jump size is the same as that of $Z$ and $Y=X+Z$ satisfies aforementioned conditions.  From the point of view of filtering theory $X$ can be considered as the  unobserved `drift' added to the martingale $Z$. The
specific choice of $X$ used in the bridge construction ensures that
this drift {\em disappears} when we consider  $Y$ in its own
filtration.

To the best of our knowledge such a bridge construction has not been
studied in the literature before. On the other hand, the analogy with
the enlargement of filtration theory for Brownian motion is
obvious. Indeed, if $Z$ is instead a Brownian motion and  we consider the problem of finding a stochastic
process $X$ so that $Y=Z+X$ is a Brownian motion in its own filtration
and $[\tv=1]=[Y_1\geq y]$ almost surely for some $y \in \bbR$ to be determined, the
solution follows easily from  the
enlargement of filtration theory. The recipe is the following: Find
the Doob-Meyer decomposition of $Z$ when its natural filtration is initially enlarged
with the random variable $[Z_1\geq y]$. Then, in the finite variation
part of this decomposition, replace $Z$ with $Y$ and $[Z_1\geq y]$ with
$[\tv=1]$ to find $X$. This recipe gives
\be \label{e:Xint}
X=\indic_{[\tilde{v}=1]}\int_0^\cdot{\partial_y}
\log p^0(Y_s,s)\,ds+\indic_{[\tilde{v}=0]}\int_0^\cdot{\partial_ y}\log(1-
p^0(Y_s,s))\,ds,
\ee
where $p^0$ is the function given in (\ref{e:kbp}).
From the insider trading point of view,  $X$ defined by  (\ref{e:Xint}) is the insider's optimal trading strategy in a Kyle model, see Remark \ref{r:kbe} in this respect.  The counterpart of these arguments in the theory of enlargement of filtrations for jump processes also exists in the literature, see \cite{Kohatsu-Higa-Yamazato}.

Yet the above recipe does not work when $Z$ is the
difference of independent Poisson processes. The problem is that the
enlargement of filtration technique gives us the  decomposition of $Z$
as a sum of a martingale and an absolutely continuous process. This is
clearly not useful for the construction that we are after, since we
want to write $Y$ as sum of $Z$ and $X$ which changes only by jumps.
The desired jump process $X$ is constructed explicitly in Section \ref{sec: bridge construction} using $[\tv=1]$ and a sequences of iid uniformly distributed random variables independent of everything else. This amounts to say that the insider uses her private information and some additional randomness from uniformly distributed random variables to construct her optimal strategy.  Moreover, we will see in Section \ref{sec: equilibrium
  existence} that, after an appropriate rescaling,  these jump processes converge weakly to $X$ given by (\ref{e:Xint})
as the intensity of the Poisson processes that constitute $Z$ increases to infinity. Note the process $X$ given in (\ref{e:Xint}) does not need any extra randomness other than the set $[\tilde{v}=1]$. This brings fore the question whether the bridge process defined in Section \ref{sec: bridge construction}  can alternatively be constructed without the aid of the extra randomness. We believe this would be a quite interesting avenue for further research.

The construction of the point process bridge $Y$ allows us to prove the existence of Glosten-Milgrom equilibrium (see Theorem \ref{thm: existence equilibrium}) which was demonstrated in \cite{Back-Baruch} via a numeric computation. In such an equilibrium the insider uses a mixed strategy to randomly submit two types of orders: one type trades in the same direction as noise trades while the other cancels noise trades by submitting opposite orders when noise trades arrive. Observing noise trades, the insider uses the uniformly distributed random variables to construct her strategy inductively. On the other hand, the construction of $Y$ invites a natural application of weak convergence theory to show Glosten-Milgrom equilibria converge weakly to Kyle equilibrium when the intensity of $Z$ increases to infinity. This convergence was first proved in \cite{Back-Baruch} under strong assumption on the convergence of value functions. Utilising the theory of weak convergence, we are able to prove the result of Back and Baruch on convergence without the additional assumptions; see Theorem \ref{thm: convergence}.

The outline of the paper is as follows. In Sections \ref{sec: model}
and \ref{sec: equilibrium HJB} we
describe the Glosten-Milgrom model and characterise its equilibrium
which is the motivation of this paper. Section \ref{sec: bridge
  construction} discusses the construction of the aforementioned point
process bridge. In Section \ref{sec: equilibrium existence} we apply
the results of Section \ref{sec: bridge
  construction} to show the existence of Glosten-Milgrom equilibria
and discuss their weak convergence.
\section{The model}\label{sec: model}

We consider a market in continuous-time for a risky asset whose
{\em fundamental value} is given by $\tv$.
The investors in this market can also trade a riskless asset at an interest rate
normalised to $0$ for simplicity.  Following \cite{Back-Baruch} we
assume that $\tv$ has two states: high and low, which correspond to two numeric representations respectively, $1$ and $0$. This fundamental value will be revealed to the market participants at time $1$ at which point
we assume the market for the risky asset will
terminate\footnote{\cite{Back-Baruch} assumes that the market has a
  random horizon defined by an independent exponential random variable. However,
  one can see that
  this distinction is not relevant by comparing our results to those
of Back and Baruch.}.

The microstructure of the market, and the interaction of market
participants, is modelled similarly as in \cite{Back-Baruch}. There are
three types of agents: noisy/liquidity traders, an informed
trader (insider), and a market maker, all of whom are risk
neutral. All the processes and random variables in this section are defined on a filtered probability space $(\Om, \cF, (\cF_t)_{t \in [0,1]}, \bbP)$ satisfying the usual conditions. We assume that $\tv$ is indeed random, i.e. $\bbP(\tv=0) \in (0,1)$.

\begin{itemize}
\item \textit{Noisy/liquidity traders} trade for liquidity reasons, and
their total demand is given by the difference of two pure jump processes $Z^B$ and $Z^S$, which represent
their cumulative buy and sell orders, respectively. As such, the net
order flow of the noise traders are given by
$Z:=Z^B-Z^S$. Noise traders only submit orders of fixed size $\delta$ every time they trade. As in \cite{Back-Baruch}, $Z^B/\delta$ and $
Z^S/\delta$ are assumed to be independent Poisson processes with constant intensity
$\beta$. Moreover, they are independent of $\tv$.
\item \textit{The informed trader} observes the market price process
  and is given the value of $\tv$ at time $0$. The net order of the
  insider is denoted by $X:=X^B-X^S$ where $X^B$ (resp. $X^S$)
  denotes the cumulative buy (resp. sell) orders of the insider.
\item A competitive \textit{market maker} observes only the total net demand process $
Y_{t}=X_t +Z_t$ and sets the price based solely on this information.  This in particular implies that the market maker's
filtration is $(\mathcal{F}_{t}^{Y})$, the minimal filtration generated by $Y$ satisfying the usual conditions. We assume that the market maker is
risk neutral and, thus, the competitiveness means that he sets the price at
$\bbE[\tv|\cF^Y_t]$ in the equilibrium.
\end{itemize}

Although the noise traders trade for liquidity reasons exogenous to this model,
the insider has the objective to maximise her expected profit out of
trading. This strategic behaviour of the insider and the pricing
mechanism set by the market maker as described above results in the
price being determined in an equilibrium. In order to define precisely
what we mean by an equilibrium between the market maker and the
insider,  we first need to establish the class of
{\em admissible} actions available to both.
\begin{defn}\label{def: rational pricing rule} A function $p:\delta \Integer \times [0,1]  \rightarrow [0,1]$ is a \emph{pricing rule} if
 \begin{enumerate}
  \item[i)] $y\mapsto p(y, t)$ is strictly increasing for each $t\in[0,1)$;
  \item[ii)] $t \mapsto p(y, t)$ is continuously differentiable for each $y \in \delta \Integer$.
 \end{enumerate}
\end{defn}
 This Markov assumption on
 the pricing functional is standard in the literature (see,
 e.g., \cite{Back}, \cite{BP} or \cite{CCD}). Given the pricing rule,
 the market maker sets the price to be $p(Y_t, t)$. It would be
 irrational for the market maker to price the asset at some value
 larger than $1$ or less than $0$ since everybody knows that the true
 value of the asset is $0$ or $1$. As we mentioned above the market
 maker is competitive so that in equilibrium the price equals
 $\expec[\tilde{v} \,|\, \F^Y_t]$. Hence, $p$ is typically $[0,1]$-valued.
 The monotonicity of $p(\cdot, t)$ implies that an increase in demand
 has a positive feedback on the asset price. Moreover, this leads the insider to fully observe the noise trades, $Z$, by simply inverting the
 price process and subtracting her own trades from it. Consequently, the insider's filtration, denoted with $\cF^I$,
 contains the
 filtration generated by $Z$ and $\tv$. We shall assume $\F^I$ satisfies the usual conditions. However, we refrain from setting $\cF^I$ equal to the filtration generated by $Z$ and initially enlarged with $\tv$ since we will only be able to show the existence of equilibrium if the insider also possess a sequence of independent random variables, which she will use in order to construct her  mixed strategy. Admissible strategy of the insider is defined as follows.
\begin{defn}\label{def: admissible}
 The strategy $(X^B, X^S; \F^I)$ is \emph{admissible}, if
 \begin{enumerate}
  \item[i)] $\cF^I$ is a filtration satisfying the usual conditions
    such that $\cF^I_t=\sigma(v, \cF^Z_t, \cH_t)$, where $\cH$ is a
    filtration independent of $v$ and $\cF^Z$.
\item[ii)] $X^B$ and $X^S$, with $X_0^B = X_0^S=0$, are
    $\cF^I$-adapted and integrable\footnote{That is, $\expec[X^B_1]$ and $\expec[X^S_1]$ are both finite.} increasing point processes with
    jump size $\delta$;
  \item[iii)] the $(\cF^I, \bbP)$-dual predictable
    projections\footnote{These are simply the predictable compensators of the
      increasing processes $X^B$ and $X^S$. See, e.g. \cite{Jacod} for
      a precise definition.}  of $X^B$ and $X^S$ are absolutely
    continuous functions of time.
 \end{enumerate}
\end{defn}
The first assumption on $\cF^I$ makes the insider's filtration part of
the equilibrium. This is to allow mixed strategies which will be
determined in equilibrium. Note that the additional information can
only  come from a source that is independent of $Z$. This implies in
particular that the insider does not have any extra information about the
future demand of the noise traders. Although we allow this additional
source of information to vary in time, in the form of filtration
$\cH$, in the equilibrium that we will compute, $\cH_t=\cH_0$ for all
$t \in [0,1]$.

We assume that the insider can only trade $\delta$-shares of the asset
in every trade like the noise traders. This is one of the underlying
assumptions of the Glosten-Milgrom model, which we keep in this paper
as well. One intuitive reason for this is that  a rational insider will never submit an order of a
different size, since this will immediately reveal her identity and
make, at least a part of, her private information public causing to
lose her comparative advantage. Moreover, in order to make this
argument rigorous one needs to make assumptions on the  pricing rule as to
how to handle the orders of sizes which are multiples of
$\delta$. One can do the pricing uniformly, i.e. every little bit of the
order is priced the same,  or different parts of the order is priced
differently as one walks up or down in an order book (see \cite{BB2} for a
discussion of such issues). However, this requires different
techniques for the analysis of optimal strategies given this
complicated nature of pricing; thus, we leave such analysis to a
future investigation.

 The third assumption on the dual predictable projections implies that $X^B$ and $X^S$ admit $\F^I$-intensities $\theta^B$ and $\theta^S$ such that $X^B -\int_0^\cdot \theta^B_s \, ds$ and $X^S - \int_0^\cdot \theta^S_s\, ds$ are $\F^I$-martingales (see \cite[Chapter 1, Theorem 3.15]{Jacod-Shiryaev}).
This assumption is technical and to ensure tractability.

Given that the insider submits orders of size $\delta$ and the
assumption that the market maker observes only the net demand yield
that when the insider submits an order at the same when an uninformed
order arrives but in the opposite direction, i.e. a trade between the
informed and uninformed occurs without a need for a market maker, this
transaction goes unnoticed by the market maker. Thus, what we are
affectively assuming is that the marke maker only becomes aware of the
transaction when there is a need for him. The assumption that the
market maker only observes the net demand is a common assumption in
market microstructure literature. In particular, it is always assumed
in Kyle type models (see, e.g. \cite{BAPT}). Henceforth, when the insider makes a
trade at the same with an uninformed trader but in an opposite
direction, we will say that the insider {\em cancels} the noise trades.

Although we allow the insider to trade at the same time with the noise
traders in the same direction,  we will see that in the equilibrium
the insider will not carry such trades. This is intuitive. does not
trade in the same direction at the same time as the uniformed trades,
but she does randomly cancel part of uninformed orders. Both actions
are required to hide her identity from the market maker.  Indeed, when
two buy orders arrive at the same time the market maker will know that
one of them is an informed trade. Therefore it would be to the
advantage of the insider  to hide her trades by submitting randomly,
but of the same size, among the uninformed trades. On the other hand,
since the market maker is not aware of the transactions which consist
in canceling noise trades, submitting an order at the same time with
the noise traders but in the opposite direction is not necessarily
suboptimal. We will in fact see that the insider does randomly cancel some
trades that are placed by the noise traders in the equilibrium.


As discussed in the last paragraphs, the insider's buy orders $X^B$ consist of three components: we denote by $X^{B,B}$ the cumulative buy orders which arrive at different time than those of $Z^B$, by $X^{B, T}$ the cumulative buy orders which arrive at the same time as some orders of $Z^B$, and by $X^{B,S}$ the cumulative buy orders which cancel some sell orders of $Z^S$. As such, the jump time of $X^{B,T}$ (resp. $X^{B,S}$) are contained in the set of jump times of $Z^B$ (resp. $Z^S$). Sell orders $X^{S,S}, X^{S, T}$, and $X^{S, B}$ are defined analogously. Therefore $X^B = X^{B,B} + X^{B, T} + X^{B,S}$ and $X^S = X^{S,S} + X^{S, T} + X^{S,B}$.

As mentioned earlier, the insider aims to maximise her expected
profit. Given an admissible trading strategy $(X^B, X^S)$ the
associated profit at time $1$ of the insider is given by
\[
\int_0^1 X_{t-}\, dp(Y_t,t)  + (\tv -p(Y_1,1))X_1.
\]
The last term appears due to a potential discrepancy between the market price and the liquidation value. Since $X$ is of finite variation, an application of integration by parts rewrites the above as
\bean
&&\int_0^1 (\tv -p(Y_t,t))\,dX^B_t -\int_0^1 (\tv -p(Y_t, t))\,dX^S_t\\
&=&\quad \int_0^1(\tv -p(Y_{t-}+\delta, t))\,dX^{B,B}_t + \int_0^1 (\tv-p(Y_{t-} + 2\delta, t)) \, dX^{B, T}_t + \int_0^1 (\tv - p(Y_{t-}, t)) \, dX^{B,S}_t\\
&& -\int_0^1 (\tv
-p(Y_{t-}-\delta, t))\,dX^{S,S}_t - \int_0^1 (\tv- p(Y_{t-}-2\delta, t)) \, dX^{S, T}_t - \int_0^1 (\tv- p(Y_{t-}, t))\, dX^{S, B}_t,
\eean
where the last line is due
to the fact that $Y$ increases by $\delta$ when $X^{B,B}$
jumps, increases by $2\delta$ when $X^{B,T}$ jumps, and is unchanged when $X^{B,S}$ and $Z^S$ jump at the same time but different directions. Similar situation goes for negative jumps of $Y$. As seen from the above formula, the profit is zero when the insider place two opposite orders as the same time, we then assume without loss of generality that insider does not do so.

Let's define
\[
a(y,t):=p(y+\delta,t) \qquad \mbox{and} \qquad b(y,t)=p(y-\delta,t).
\]
Then, the expected profit of the insider conditional on her information equals
\begin{equation}\label{eq: insider problem}
\begin{split}
  &\expec_\prob \left[\int_0^1 (\tv - a(Y_{t-},t)) \,dX^{B,B}_t + \int_0^1 (\tv - a(Y_{t-} + \delta, t)) \, dX^{B, T}_t + \int_0^1 (\tv - p(Y_{t-}, t)) \, dX^{B,S}_t \right. \\
  & \quad\left.\left. - \int_0^1 (\tv- b(Y_{t-},t)) \,dX^{S,S}_t - \int_0^1 (\tv- p(Y_{t-}-\delta, t)) \, dX^{S,T}_t - \int_0^1 (\tv- p(Y_{t-}, t)) \, dX^{S,B} \right| \tv\right].
\end{split}
\end{equation}
Note that the assumption $\bbE[X^B_1] < \infty$ implies $\bbE[X^B_1|\tv] < \infty$ as well since $\bbE[X^B_1]=\bbE[X^B_1|\tv=1]\bbP[\tv=1]+ \bbE[X^B_1|\tv=0]\bbP[\tv=0]$, and $\bbP[\tv=0] \in (0,1)$.  Similarly, $\bbE[X^S_1|\tv] < \infty$, too. Thus, the above expectation will be finite as soon as we assume
that the pricing rule is rational in the sense that it assigns a price
to the asset between $0$ and $1$. This will be part of the definition
of equilibrium, which will be made precise below. As seen from the above formulation, when price moves, one buys (resp. sells) at a price  $a(y,t)$ (resp. $b(y,t)$), where $y$ is the cumulative order right before such trade. Thus, $a(y,t)$ (resp. $b(y,t)$) can be viewed as the ask (resp. bid) price.


Our goal is to find an equilibrium between the market maker and the
insider in the following fashion:
\begin{defn}\label{def: equilibrium}
 A \emph{Glosten-Milgrom equilibrium} is a quadruplet $(p, X^B, X^S, \F^I)$
 such that
\begin{itemize}
\item[i)] given $(X^B, X^S; \F^I)$, $p$ is a {\em rational} pricing rule, i.e., $p(Y_t, t)= \expec[\tv \,|\, \F^Y_t]$ for $t\in[0,1]$;
\item[ii)] given $p$, $(X^B, X^S; \F^I)$ is an admissible strategy
maximising  \eqref{eq: insider problem}.
\end{itemize}
\end{defn}

Recall that  $\tv$ takes only two values by assumption. In view of
this specification we will often call the insider in the sequel of \emph{high type} when $\tv=1$ and \emph{low type} when $\tv=0$.

\section{Characterisation of equilibrium}\label{sec: equilibrium HJB}
Before we give a characterisation of equilibrium, we will provide some
heuristics. Due to the Markov structure of the pricing rule, we will
define the informed trader's value function and derive, via a
heuristic argument, the associated HJB equation. Definition \ref{def:
  admissible} ii) implies that the $\F^I$-dual predictable projection
of $X^{i,j}$, $i\in \{B, S\}$ and $j\in \{B, S, T\}$, is of the form
$\delta \int_0^\cdot \theta^{i,j}_s \, ds$ so that $X^{i,j} - \delta
\int_0^\cdot \theta^{i,j}_s \, ds$ defines an
$\F^I$-martingale. Observe that since the set of jumps times of $X^{B,
  S}$ and $X^{S, T}$ (resp. $X^{S, B}$ and $X^{B, T}$) is contained in
the set of jump times of $Z^S$ (resp. $Z^B$), we necessarily have
$\theta^{B, S} + \theta^{S, T}\leq \beta$ (resp. $\theta^{S, B} +
\theta^{B, T}\leq \beta$). Moreover, Definition \ref{def: equilibrium}
i) implies that $p$ takes values in $[0,1]$,  hence  both bid and ask
prices are $[0,1]$-valued by definition. Therefore, $\int_0^\cdot
(\tv- a(Y_{u-}, u)) (dX^{B,B}_u - \delta \theta^{B,B}_u \, du)$ is an
$\F^I$-martingale (see \cite[Chapter 1, T6]{Bremaud}). Arguing
similarly with the other terms, the expected profit \eqref{eq: insider problem} can then be expressed as
\[
\begin{split}
 &\delta \,\expec_\prob\left[\left.\int_0^1 (\tv - p(Y_{u-}+\delta, u)) \theta^{B,B}_u \,du +\int_0^1 (\tv - p(Y_{u-}+2\delta, u)) \theta^{B,T}_u\,du +\int_0^1 (\tv-p(Y_{u-}, u)) \theta^{B,S}_u \,du\right.\right.\\
 & \hspace{5mm}\left.\left.- \int_0^1 (\tv - p(Y_{u-}-\delta, u)) \theta^{S,S}_u \, du -\int_0^1 (\tv- p(Y_{u-}-2\delta, u)) \theta^{S,T}_u \, du -\int_0^1 (\tv- p(Y_{u-}, u)) \theta^{S,B}_u \,du \right| \tv\right].
\end{split}
\]
This motivates us to define the following value function for the informed trader:
\[
\begin{split}
 & V(\tv, y, t) = \sup_{\theta^{i,j}; \,i\in\{B,S\}, j\in \{B, S, T\}} \\
 &\delta\, \expec_\prob\left[\int_t^1 (\tv - p(Y_{u-}+\delta, u)) \theta^{B,B}_u \, du +  \int_t^1 (\tv - p(Y_{u-}+2\delta, u)) \theta^{B, T}_u + \int_t^1 (\tv- p(Y_{u-}, u)) \theta^{B,S}_u \, du\right.\\
 &\left.\left.- \int_t^1 (\tv - p(Y_{u-}-\delta, u)) \theta^{S,S}_u \,du - \int_t^1 (\tv- p(Y_{u-}-2\delta, u)) \theta^{S,T}_u \, du
 - \int_t^1 (\tv-p(Y_{u-}, u)) \theta^{S,B}_u \, du \right| Y_t = y, \tv\right],
\end{split}
\]
for $\tv\in\{0, 1\}$, $t\in [0,1)$, and $y\in \delta \Integer$. The terminal value of $V$ at $1$ can be defined via the left limit $V(\tv, y, 1) := \lim_{t\uparrow 1} V(\tv, y, t)$. As we will see in Remark \ref{rem: boundary layer} below, $V(\tv, y, 1)$ is not always zero.

Recall that $Y= X+Z$ so that if one defines $Y^B = X^{B,B} + X^{B, T} + Z^B - X^{S, B}$ and $Y^S = X^{S,S} + X^{S, T} + Z^S - X^{B,S}$, then it is easy to see that
$(Y^B_t - \delta \int_0^t (\beta - \theta^{B,T}_s - \theta^{S, B}_s)\,ds - \delta\int_0^t \theta^{B,B}_s \, ds - 2 \delta \int_0^t \theta^{B, T}_s \, ds)$ and
$(Y^S_t - \delta\int_0^t (\beta-\theta^{S,T}_s- \theta^{B, S})\,ds - \delta\int_0^t \theta^{S,S}_s \, ds - 2 \delta \int_0^t \theta^{S, T}_s \, ds)$
are $\F^I$-martingales.  Thus, applying Ito's formula to $V(\tv, Y_t, t)$ yields the following formal HJB equation (the variable $\tv$ is omitted in $V$ for simplicity of notation) in view of the standard dynamic programming arguments:
\begin{equation}\label{eq: HJB V}
\begin{split}
 0=  V_t &+ \pare{V(y+\delta, t) - 2V(y, t) + V(y-\delta, t)} \beta  \\
 &+\sup_{\theta^{B,B}\geq 0} \bra{V(y+\delta, t) - V(y, t) + \pare{\tv- p(y+\delta, t)}\delta} \theta^{B,B} \\
 &+ \sup_{\theta^{B,T}\geq 0} \bra{V(y+2\delta, t) - V(y+\delta, t) + \delta (\tv - p(y+2\delta, t))} \theta^{B, T}\\
 &+ \sup_{\theta^{B,S}\geq 0} \bra{V(y, t) -V(y-\delta, t) + (\tv - p(y,t))\delta}\theta^{B,S} \\
 &+ \sup_{\theta^{S,S} \geq 0} \bra{V(y-\delta, t) - V(y, t) - \pare{\tv- p(y-\delta, t)}\delta} \theta^{S,S} \\
 &+ \sup_{\theta^{S,T}\geq 0} \bra{V(y-2\delta, t) - V(y- \delta, t) - \delta (\tv- p(y-2\delta, t))} \theta^{S, T}\\
 &+ \sup_{\theta^{S,B} \geq 0} \bra{V(y, t) - V(y+\delta, t) - (\tv- p(y,t))\delta} \theta^{S,B}, \quad (y, t)\in \delta \Integer \times [0,1).
\end{split}
\end{equation}
The optimiser $(\theta^{i,j}; i\in \{B,S\} \text{ and } j\in \{B, S, T\})$ in the previous equation is expected to be the $\F^I$-intensities of the insider's optimal strategy $(X^{i,j})$ when the order size is normalised to $1$.

Notice that all maximisations in \eqref{eq: HJB V} are linear in $\theta$. Therefore \eqref{eq: HJB V} reduces to the following system:
\begin{equation}\label{eq: HJB system V}
\begin{split}
 V_t + \pare{V(y+\delta, t) - 2V(y, t) + V(y-\delta, t)} \beta&=0,\\
 V(y+\delta, t) - V(y, t) + (\tv - p(y+\delta, t)) \delta &\leq 0,\\
 V(y-\delta, t) - V(y, t) - (\tv - p(y-\delta, t)) \delta &\leq 0, \quad (y, t) \in \delta \Integer \times [0,1).
\end{split}
\end{equation}
Here the first inequality corresponds to the maximisation in $\theta^{B, j}$; while the second inequality corresponds to the maximisation in $\theta^{S, j}$, $j\in \{B, S, T\}$.
Let's denote the optimisers in \eqref{eq: HJB V} with
$(\theta^{i,j}(y,t); (y,t)\in\delta \Integer \times [0,1))$, $i\in
\{B,S\}$ and $j\in \{B, S, T\}$. Observe that the first inequality in \eqref{eq: HJB system
  V} can be strict only if $\theta^{B,B}(y,t) = \theta^{B,S}(y+\delta,
t)=\theta^{B, T}(y-\delta, t)=0$. Similarly, the second inequality can be strict only if
$\theta^{S,S}(y,t) = \theta^{S,B}(y-\delta, t)=\theta^{S,T}(y+\delta, t)=0$. We will see later
that the optimal $\theta^{B,B}$ and $\theta^{B,S}$ are never $0$ for
the high type insider meanwhile $\theta^{S,S}$ and $\theta^{S,B}$ are
never $0$ for the low type. Therefore the first inequality in
\eqref{eq: HJB system V} is actually an equality when $\tv=1$ and the
second inequality is an equality when $\tv=0$. Economically speaking,
these equalities imply that at every instant of time there is a
non-zero probability that a high type insider will make a buy order by
either contributing to uninformed buy orders or canceling uninformed
sell orders, and the low type insider will do the opposite. Such
actions are certainly reasonable for the insider. Indeed, a high type
insider will reveal her information gradually and keep the market
price strictly less than $1$. Recall that $p$ is a martingale bounded
by $1$, so once it hits $1$, it will be stopped at that
level. Therefore, since there is always a strictly positive difference
between the true price, which is $1$ in this case, and the market price, the insider will always want to take advantage of this discrepancy and buy with positive probability since the asset is undervalued by the market. The situation for the low type is similar.

In view of the previous discussion, let's consider the  following system:
\begin{equation}\label{eq: HJB system H}\tag{HJB-H}
\begin{split}
 &V^H_t + \pare{V^H(y+\delta, t) - 2V^H(y, t) + V^H(y-\delta, t)} \beta =0,\\
 &V^H(y+\delta, t) - V^H(y, t) + (1- p(y+\delta, t)) \delta =0;
\end{split}
\end{equation}
\begin{equation}\label{eq: HJB system L}\tag{HJB-L}
\begin{split}
 &V^L_t + \pare{V^L(y+\delta, t) - 2V^L(y, t) + V^L(y-\delta, t)} \beta =0,\\
 &V^L(y-\delta, t) - V^L(y, t) + p(y-\delta, t) \delta=0,
\end{split}
\end{equation}
for $(y, t) \in \delta \Integer \times [0,1)$. We expect that $V^H(y,
t) = V(1, y, t)$ and $V^L(y, t) = V(0, y, t)$.  The next lemma will
construct solutions to the above system and will be useful in solving
the insider's optimisation problem. However, before the statement and
the proof of this lemma we need to introduce a class of functions
satisfying certain boundary conditions  and differential equations. We will nevertheless denote them with $p$ since, as we shall see later, they will  appear in the equilibrium as pricing rules for the market maker.

To this end, for each $z \in \delta \Integer $, let
\begin{equation}\label{eq: def P}
 P^z(y) := \left\{\begin{array}{ll}0, & y<z\\ 1, & y\geq z\end{array}\right.,
\end{equation}
and define
 \begin{equation}\label{eq: pz}
  p^z(y, t) := \expec_\prob[P^z(Z_1) \,|\, Z_t =y].
 \end{equation}
Observe that $Z/\delta$ is the difference of two independent Poisson processes. The Markov property implies\footnote{The Markov property of $Z$ implies that $\prob(Z_1 = \tilde{z} \,|\, Z_t=y)$ satisfies $p_t + (p(y+\delta, t) - 2p(y, t) + p(y-\delta, t))\beta=0$. Therefore summing up the previous equation for different $z$ induces that $\sum_{\delta \Integer \ni \tilde{z} \geq z} \partial_t \prob(Z_1 = \tilde{z} \,|\, Z_t =y)$ is finite. Hence Fubini's theorem implies that the previous sum is exactly $\partial_t p^z$ and $p^z$ solves \eqref{eq: pz eqn}.} that $p^z$ satisfies
\begin{equation}\label{eq: pz eqn}
 \begin{split}
  &p^z_t + \pare{p^z(y+\delta, t) -2p^z(y, t) + p^z(y-\delta, t)} \beta=0, \quad (y,t)\in \delta \Integer \times [0,1),\\
  &p^z(y,1) = P^z(y),  \hspace{5.8cm} y\in \delta \Integer.
 \end{split}
 \end{equation}

\begin{lem}\label{lem: construction H L} Let $p^z$ be defined by (\ref{eq: pz}) for some fixed $ z \in \delta \Integer $ and define
 \[
  H(y, 1):= \delta \sum_{j=\frac{y}{\delta}}^{\frac{z-\delta}{\delta}} (1- A(\delta j)), \quad L(y, 1):= \delta \sum_{j=\frac{z}{\delta}}^{\frac{y}{\delta}} B(\delta j), \quad y\in \delta \Integer,
 \]
 where $A(y):= P^z(y+\delta)$, $B(y):= P^z(y-\delta)$, and $\sum_{j=m}^n \alpha_j=- \sum_{j=n}^m \alpha_j$ by convention whenever $m>n$. Then, both $H(\cdot, 1)$ and $L(\cdot, 1)$ are nonnegative and the following equivalences hold:
 \[
 H(y, 1) =0 \Longleftrightarrow A(y)=1 \Longleftrightarrow y\geq z- \delta, \quad L(y,1) =0 \Longleftrightarrow B(y)=0 \Longleftrightarrow y<z + \delta.
 \]
 Moreover,
 \begin{eqnarray}
  H(y, t)&:=& H(y, 1) +  \delta \beta\int_t^1 \pare{p^z(y+\delta, u) - p^z(y,u)} du \quad \text{ and } \label{eq: H def}\\
  L(y,t)&:=& L(y,1) +  \delta \beta \int_t^1 \pare{p^z(y, u) -p^z(y-\delta, u)} du \label{eq: L def}
 \end{eqnarray}
 solve \eqref{eq: HJB system H} and \eqref{eq: HJB system L} respectively.
\end{lem}

\begin{proof}
 Statements regarding $H(y,1)$ and $L(y,1)$ directly follow from the definitions. We will next show that $H$ satsifies \eqref{eq: HJB system H}. Analogous statement for $L$ can be proven similarly. First observe that
 \[
 H(y+\delta, 1) - H(y, 1) = -\delta + \delta A(y) = -\delta + \delta P^z(y+\delta).
 \]
 Thus,
 \bea
  H(y+\delta, t) - H(y, t) &=& H(y+\delta, 1) - H(y, 1)+ \delta \beta \int_t^1 \pare{p^z(y+2\delta, u) - 2p^z(y+\delta, u) + p^z(y, u)} du \nn \\
  &=& \delta \pare{p^z(y+\delta, t) -1},  \label{eq: H sec eqn}
 \eea
 where \eqref{eq: pz eqn} is used to obtain the last line. This proves the second equation in \eqref{eq: HJB system H}.

 Next, it follows from the definition of $H$ that
 \[
 H_t(y, t)+ \delta \beta \pare{p^z(y+\delta, t) - p^z(y, t)}=0.
 \]
 However, iterating  \eqref{eq: H sec eqn} yields
 \bean
 H(y+\delta,t)+H(y-\delta,t)-2H(y,t)&=&H(y+\delta,t)- H(y, t)-\left(H(y, t)-H(y-\delta,t)\right)\\
  &=&\delta\pare{p^z(y+\delta, t) -p^z(y,t)},
  \eean
  and, hence, the claim.
\end{proof}
Given a pricing rule, let us describe insider's optimal strategies.
\begin{prop}\label{p: UB value func}
 Suppose that the market maker chooses $p^z$ as a pricing rule, where $z$ is fixed and $p^z$ is as defined in (\ref{eq: pz}). Then, the following holds:
 \begin{enumerate}
  \item[i)] When $\tv=1$,  $(X^{B}, X^{S}; \F^I)$ is an optimal strategy if and only if $Y_1 \geq z-\delta$ and $X^{S,j} = 0$, $j=\{B,S,T\}$.
  \item[ii)] When $\tv=0$, $(X^{B}, X^{S}; \F^I)$ is an optimal strategy if and only if $Y_1 <z+\delta$ and $X^{B,j} =0$, $j=\{B,S,T\}$.
 \end{enumerate}
 When the previous condition holds for $\tv=1$ (resp. $\tv=0$), $v(1, y, t) = H(y,t)$ (resp. $v(0, y, t) = L(y,t)$) for $(y,t) \in \delta \Integer \times [0,1]$.
\end{prop}
\begin{rem}\label{rem: boundary layer}
 Recall that $V(\tv, y, 1):=\lim_{t\uparrow 1} V(\tv, y, t)$. Lemma \ref{lem: construction H L} and Proposition \ref{p: UB value func} combined implies that $V(\tv, y, 1) \geq 0$. It is only zero when $A(y) = 1$ for the high type and $B(y)=0$ for the low type.
\end{rem}

\begin{proof}
  The statements for $\tv=1$ case will be proved. Similar arguments can be applied in order  to prove the statement regarding $\tv=0$. Fix $(y, t) \in \delta \Integer \times [0,1)$. For any admissible trading strategy $(X^{i,j}; i\in\{B,S\})$ and $j\in \{B, S, T\}$ with associated $\cF^I$-intensities $(\delta \theta^{i,j}; i\in \{B, S\} \text{ and } j\in \{B, S, T\})$, applying Ito's formula to $H(Y_\cdot, \cdot)$ and utilizing Lemma \ref{lem: construction H L}, we obtain
   \bean
  && H(Y_1, 1)\\
  &=& H(y, t) + \int_t^1 H_t(Y_{u-}, u) du \\
  && + \int_t^1 \pare{H(Y_{u-}+\delta, u) - H(Y_{u-}, u)} (\beta- \theta^{B, T}_u - \theta^{S,B}_u) \,du + \int_t^1 \pare{H(Y_{u-}+\delta, u) - H(Y_{u-}, u)} \theta^{B,B}_u\, du \\
  &&+ \int_t^1 \pare{H(Y_{u-} + 2\delta, t) - H(Y_{u-}, u)} \theta^{B,T}_u \, du\\
  &&+ \int_t^1 \pare{H(Y_{u-}-\delta, u) - H(Y_{u-}, u)} (\beta- \theta^{S,T} - \theta^{B,S}_u) \,du + \int_t^1 \pare{H(Y_{u-}-\delta, u) -H(Y_{u-}, u)} \theta^{S,S}_u \, du\\
  &&+ \int_t^1 \pare{H(Y_{u-}- 2\delta, u) - H(Y_{u-}, u)} \theta^{S, T}_u\, du + M_1-M_t\nn \\
  &=& H(y, t) \\
  && - \int_t^1 \pare{H(Y_{u-}+\delta, u) - H(Y_{u-}, u)} \theta^{S,B}_u \,du + \int_t^1 \pare{H(Y_{u-}+\delta, u) - H(Y_{u-}, u)} \theta^{B,B}_u\, du \\
  &&+ \int_t^1 \pare{H(Y_{u-} + 2\delta, t) - H(Y_{u-}+\delta, u)} \theta^{B,T}_u \, du\\
  &&- \int_t^1 \pare{H(Y_{u-}-\delta, u) - H(Y_{u-}, u)}  \theta^{B,S}_u \,du + \int_t^1 \pare{H(Y_{u-}-\delta, u) -H(Y_{u-}, u)} \theta^{S,S}_u \, du\\
  &&+ \int_t^1 \pare{H(Y_{u-}- 2\delta, u) - H(Y_{u-}-\delta, u)} \theta^{S, T}_u\, du+ M_1-M_t\nn \\
  &=& H(y, t) + \delta \int_t^1 \pare{p(Y_{u-}+\delta, u)-1} \theta^{B,B}_u\, du + \delta \int_t^1 \pare{1-p(Y_{u-}, u)} \theta^{S,S}_u \, du\\
  &&-\delta \int_t^1 \pare{p(Y_{u-}+\delta, u)-1} \theta^{S,B}_u \,du -\delta \int_t^1 \pare{1-p(Y_{u-}, u)} \theta^{B,S}_u \, du\\
  && - \delta \int_t^1 (1-p(Y_{u-}+ 2\delta, u)) \,\theta^{B,T}_u \, du  + \delta \int_t^1 (1- p(Y_{u-} - \delta, u)) \,\theta^{S, T}_u \, du + M_1-M_t.
 \eean
 Here $M$ contains $\int_0^\cdot (p(Y_{u-}+\delta, u) -1) (dX^{B,B}_u -\delta \theta^{B,B}_u du)$ and similar processes, which are all $\F^I$-martingales due to the bounded integrand and the martingale property of $X^{i,j}- \delta \int_0^\cdot \theta^{i,j}_u\, du$ for $i\in \{B, S\}$ and $j\in \{B,S,T\}$ (see \cite[Chapter 1, T6]{Bremaud}). Thus, on $[\tv=1]$
  \bean
 &&\delta\int_t^1 \left(1 -p(Y_{u-}+\delta,u)\right)\theta^{B,B}_u\,du +\delta \int_t^1 \pare{1-p(Y_{u-}, u)} \theta^{B,S}_u \, du + \delta \int_t^1 (1-p(Y_{u-}+2\delta, u)) \theta^{B, T}_u\, du \\
 &&-\delta\int_t^1 \left(1 -p(Y_{u-}-\delta, u)\right)\theta^{S,S}_u\,du - \delta\int_t^1 \pare{1-p(Y_{u-}, u)} \,\theta^{S, B}_u \, du - \delta \int_t^1 (1-p(Y_{u-}-2\delta, u))\, \theta^{S,T}_u \, du\\
 &=&M_1 -M_t -H(Y_1,1)+H(y,t)\\
 && -\delta\int_t^1 \pare{p(Y_{u-},u)-p(Y_{u-}-\delta,u)}\theta^{S,S}_u\,du - \delta \int_t^1 \pare{p(Y_{u-}+\delta, u) - p(Y_{u-}, u)} \theta^{S,B}_u \, du\\
 && - \delta\int_t^1 (p(Y_{u-}-\delta, u) - p(Y_{u-}-2\delta, u)) \,\theta^{S,T}_u \, du.
 \eean
  Observe that the left side of the above equality is the wealth of the insider. Moreover, since $H\geq 0$ and $p$ is strictly increasing in $y$, the expected wealth, conditioned on $\F^I_t$, is maximised when $H(Y_1,1)=0$ $\prob$-a.s., $\theta^{S,S}$, $\theta^{S,T}$, and $\theta^{S,B}$ are identically zero. However, in view of Lemma \ref{lem: construction H L}, $H(Y_1,1)=0$ if and only if $Y_1 \geq z-\delta$.
\end{proof}
We are now ready to state the conditions for equilibrium.
\begin{thm} \label{t:equilibrium}
$(p, X^B, X^S, \F^I)$ is a Glosten-Milgrom equilibrium if there exists a $y_\delta\in \delta \Integer$ such that
\begin{itemize}
\item[i)] $[Y_1\geq y_\delta] = [\tv=1]$ $\prob$-a.s.;
\item[ii)]$p=p^{y_{\delta}}$ which is defined by (\ref{eq: pz});
\item[iii)] $(X^B, X^S; \F^I)$ is an admissible strategy such that $Y=Z
  +X^B-X^S=Y^B-Y^S$ where $Y^B/\delta$ and $Y^S/\delta$ are
  independent, $\cF^Y$-adapted Poisson processes with  common
  intensity $\beta$, and $X^S \equiv 0$ (resp. $X^B\equiv 0$) on $[\tv=1]$
  (resp. $[\tv=0]$).\end{itemize}
\end{thm}
\begin{proof}
Given the pricing rule $p=p^{y_\delta}$,
Proposition \ref{p: UB value func} implies that $(X^B, X^S; \F^I)$ is optimal because $[Y_1 \geq y_\delta] = [\tv=1]$ $\prob$-a.s. and $X^S\equiv 0$ (resp. $X^B\equiv 0$) on $[\tv=1]$ (resp. $[\tv=0]$). Thus it remans to show $p^{y_\delta}$ is a rational pricing rule given $(X^B, X^S; \F^I)$. Indeed, since $Y$ and $Z$ have the same distribution, it follows from  (\ref{eq: pz}) and the Markov property of $Y$  that
  $\expec_\prob[\tv|\cF^Y_t]=\prob[Y_1\geq
  y_\delta|\cF^Y_t]=p^{y_\delta}(Y_t,t)$ for $t\in[0,1]$.
\end{proof}

\begin{rem}\label{rem: no top up}
 Theorem \ref{t:equilibrium} iii) necessarily requires that $X^{B,T}\equiv 0$ (resp. $X^{S, T}\equiv 0$) on $[\tv=1]$ (resp. $[\tv=0]$) since it implies that the jumps occur with magnitude $\delta$ only. Recall from the proof of Proposition \ref{p: UB value func} that this is not a requirement for optimality from the point of view of insider. Rather, the insider chooses not to trade at the same time and in the same direction with the noise traders in order to make it possible that there is a rational pricing rule that  the market maker can choose.
\end{rem}
The equilibrium given in the above theorem is another manifestation of
{\em inconspicuous trade theorem} commonly observed in the insider
trading literature (see, e.g., \cite{Kyle}, \cite{Back}, \cite{BP}, etc.). Indeed, when the insider is trading optimally in the above
equilibrium, the distribution of the net order process is the same as
that of the net orders of the noise traders, i.e. the insider is able
to hide her trades among the noise trades. However, the private
information is fully, albeit gradually, revealed to the public since
$\tv \in \cF^Y_1$.
 We will construct an admissible strategy satisfying conditions above and show the existence of Glosten-Milgrom equilibrium in the following
section.
\begin{rem}  Proposition \ref{p: UB value func}  and Theorem
  \ref{t:equilibrium} indicate that `bluffing' strategies selling for the high-type
  and buying for the  low-type are sub-optimal. This is in contrast to the  results in
  \cite{Back-Baruch}, which use numeric computations to suggest such bluffing might be optimal.
\end{rem}

\section{Construction of a point process bridge}\label{sec: bridge construction}
As seen in Theorem \ref{t:equilibrium} we are interested in the
construction of a process $Y= Z +X^B-X^S$ such that, in its natural
filtration, $Y=Y^B-Y^S$ such that $Y^B/\delta$ and $Y^S/\delta$ are independent Poisson processes
with intensity $\beta$. To this end, we will construct explicitly  a
process $Y$ on some $(\Om, \cF, (\cF_t)_{t \in [0,1]}, \bbP)$ such that
\begin{equation}\label{eq: Y decomp}
 Y= Z^B - Z^S + X^B \, \indic_{I} - X^S\, \indic_{I^c},
\end{equation}
where $I\in \F_0$ with specified probability, $X^B$ and $X^S$ are two point processes and $Z/\delta$ is
$\cF$-adapted and is the difference of two independent Poisson
processes with intensity $\beta$. In particular, $I$ is independent of $Z$ since $Z$ has independent increments and $Z_0=0$. The set
$I$ is to be associated with the set $[\tv=1]$. In order to comply
with the conditions of the equilibrium described in the last
section, we will further require $[Y_1 \geq y_{\delta}]=I$ $\prob$-a.s. for a given
suitable $y_\delta$. Since $Y$ is expected to have the same distribution as $Z$, the previous condition necessitates
$\prob(I)=\prob(Z_1\geq y_\delta)$. During the construction of the
probability space and the process $Y$, we will take $\delta=1$ without
loss of generality since all the processes can be scaled by
$\delta$ to construct the process we are after.

In order to construct such a process we first need to determine its intensity. Since $Y$ would behave like $Z$ in its own filtration, we can view, in the sense of equality in distributions, the decomposition in \eqref{eq: Y decomp} as that of $Z$ when its own filtration is initially enlarged with the random variable $\indic_{[Z_1 \geq y_1]}$. Thus, the intensity of $Y$ will be that of $Z$ in this enlarged filtration.

Let $(\bbD([0,1],\Integer), {\F^1}, (\F^1_t)_{t\in[0,1]}, {\prob^1})$ be the canonical space where $\bbD([0,1],\Integer)$ is $\Integer$-valued \cadlag\, functions, $\prob^1$ is a probability measure under which $Z^B$ and $Z^S$ are independent Poisson processes with intensities $\beta$, $(\F^1_t)_{t\in[0,1]}$ is the minimal filtration generated by $Z^B$ and $Z^S$ satisfying the usual conditions, and ${\F^1}=\bigvee_{t \in [0,1]}\F^1_t$. Let's denote with  $(\mathcal{G}^1_t)_{t\in[0,1]}$ the filtration $(\F^1_t)_{t\in[0,1]}$ enlarged with the random variable $\indic_{[Z_1 \geq y_1]}$.

In order to find the ${\mathcal{G}^1}$-intensity of $Z$, we will use a standard enlargement of filtration argument which can be found, e.g., in \cite{Mansuy-Yor}. To this end, let $h: [0,1] \times \Integer \mapsto [0,1]$ be the function defined by
\begin{equation}\label{eq: def h}
 h(z,t):= \bbP^1[{Z}_1 \geq y_1 \,|\, {Z}_t = z].
\end{equation}
Note that $h$ is strictly positive on $[0,1) \times \Integer$. Moreover since $(h( {Z}_t,t))_{t\in [0,1]}$ is an ${\F^1}$-martingale, Ito's formula yields
\begin{equation}\label{eq: eqn h}
 h_t(z,t) + \beta \pare{h(z+1,t) + h(z-1,t) - 2h(z,t)} =0, \quad (t, z) \in [0,1) \times \Integer.
\end{equation}

\begin{lem}\label{lem: G intensity}
 The ${\mathcal{G}^1}$-intensities of $Z^B$ and $Z^S$ at $t\in[0,1)$ are given by
 \begin{align*}
  &\indic_{[{Z}_1 \geq y_1]} \beta \frac{h({Z}_{t-} +1,t)}{h({Z}_{t-},t)} + \indic_{[{Z}_1 < y_1]} \beta \frac{1-h({Z}_{t-}+1,t)}{1-h({Z}_{t-},t)}, \\
  &\indic_{[{Z}_1 \geq y_1]} \beta \frac{h({Z}_{t-} -1,t)}{h({Z}_{t-},t)} + \indic_{[{Z}_1 < y_1]} \beta \frac{1-h({Z}_{t-}-1,t)}{1-h({Z}_{t-},t)},
 \end{align*}
 respectively.
\end{lem}

\begin{proof}
 We will only calculate the intensity for $Z^B$. The intensity of $Z^S$ can be obtained similarly. All expectations are taken under $\bbP^1$ throughout this proof. For $s\leq t< 1$, take an arbitrary $E\in {\F}^1_s$ and denote $M^B_t := Z^B_t - \beta t$. The definition of $h$ and the $\F$-martingale property of $M^B$ imply
 \[
\begin{split}
E\bra{(M^B_t - M^B_s) \indic_E \indic_{[Z_1 \geq y_1]} }&= E\bra{(M^B_t - M^B_s) \indic_E h({Z}_t,t)} \\
&=  E\bra{\indic_E \pare{\langle M^B, h({Z}_{\cdot},\cdot)\rangle_t - \langle M^B, h({Z}_{\cdot},\cdot)\rangle_s}}\\
  &= E\bra{\indic_E \int_s^t \pare{h(Z_{r-}+1,r) - h({Z}_{r-},r)} \beta \,dr}\\
  &= E\bra{\indic_E \int_s^t \indic_{[{Z}_1 \geq y_1]}  \frac{h( {Z}_{r-}+1,r) - h( {Z}_{r-},r)}{h({Z}_{r-},r)} \beta \, dr}.
\end{split}
\]
 Since $\bbP^1({Z}_1 < \delta \,|\,{Z}_t = z) = 1- h(z,t)$, similar computations yield
 \[
 E\bra{(M^B_t - M^B_s) \indic_E \indic_{[{Z}_1 < y_1]} } = E\bra{\indic_E \int_s^t \indic_{[{Z}_1 < y_1]}  \frac{h( {Z}_{r-},r) - h( {Z}_{r-}+1,r)}{1-h({Z}_{r-},r)} \beta \,dr}.
 \]
 These computations imply that
 \[
  M^B - \int_0^\cdot \indic_{[{Z}_1 \geq y_1]}  \frac{h({Z}_{r-}+1,r) - h( {Z}_{r-},r)}{h( {Z}_{r-},r)} \beta \, dr - \int_0^\cdot \indic_{[{Z}_1 < y_1]}  \frac{h( {Z}_{r-},r) - h( {Z}_{r-}+1,r)}{1-h( {Z}_{r-},r)} \beta \,dr
 \]
 defines a ${\mathcal{G}^1}$-martingale. Therefore the ${\mathcal{G}^1}$-intensity of $Z^B$ follows from $Z^B_t = M^B_t + \beta t$.
\end{proof}

In what follows, given $I \in \F_0$ and $h$ as in \eqref{eq: def h} such that $\prob(I) = h(0,0)$, $X^B$ on $I$ and $X^S$ on $I^c$ will be constructed so that $Y$ matches the intensities given in the above lemma. As a result, Proposition \ref{p: Y property} ensures $I=[Y_1 \geq y_1]$ $\prob$-a.s., which is what we are after. We will focus on the construction of $X^B$ on $I$ in what follows. By symmetry, $X^S$ on $I^c$ can be constructed by the same method but applied to $-Z$ and $-y_1$.

Recall that one of the goals of the process $X^B$ on $I$ is to make sure that $Y_1$ ends up at a value larger than or equal to $y_1$. In order to achieve this goal $X^B$ will have to add some jumps in addition to the jumps coming from $Z^B$. However, this by itself won't be enough since $Z^S$ will make $Y$ jump downward. Thus, $X^B$ will also need to cancel some of downwards jumps coming from $Z^S$. Of course, there are many ways in which $X^B$ achieves this goal. However, $Y$ is required to have the same distribution as $Z$. We will see in Proposition \ref{p: Y property} that this distribution requirement will also be satisfied once $Y$ has the correct intensity given by Lemma \ref{lem: G intensity}.

As described above $X^B$ will consist of two components $X^{B,B}$ and $X^{B,S}$, where $X^{B,B}$ complements jumps of $Z^B$ and $X^{B,S}$ cancels some jumps of $Z^S$. Let's denote by $(\tau_i)_{i\geq 1}$ the sequence of jump times for the $Y$ process we wish to construct. These stopping times will be constructed inductively as follows. Given $\tau_{i-1}<1$, $\tau_i$ is the minimum of the following three random times:
\begin{enumerate}
 \item[i)] the next jump of $Z^B$,
 \item[ii)] the next jump of $X^{B,B}$,
 \item[iii)] the next jump of $Z^S$ which is not cancelled by a jump of $X^{B,S}$.
\end{enumerate}
Here $X^{B,B}$ and $X^{B,S}$ are constructed so that $Y^B = Z^B + X^{B,B}$ and $Y^S = Z^S - X^{B,S}$ have the required $\F$-intensities on $I$. To achieve all these aims simultaneously, when the $(i-1)^{th}$ jump of $Y$ happens before time $1$, we will generate random variables $\nu_i$ and another sequence of Bernoulli random variables $(\xi_{j, i})_{j\geq 1}$ to determine the next jump of $Y$. In the context of the informed trader trying to make a decision, construction of $X^B$ corresponds to the following pattern: place a buy order  at time $\nu_i$ unless the next buy order from the uninformed trader arrives before $\nu_i$ and also buy at every sell order of the uninformed trader until $\xi_{j, i}=1$ for the first time.

We will now make this intuitive construction rigorous. In order to
perform the subsequent construction, we must assume that the
filtered probability space $(\Om, \cF, (\cF_t)_{t \in [0,1]},\bbP)$ is
large enough so that  there exist $I\in \F_0$ with $\prob(I) = h(0,0)$ and two independent sequences of iid
$\cF$-measurable random variables $(\eta_i)_{i\geq 1}$ and
$(\zeta_i)_{i\geq 1}$ with uniform distribution on $[0, 1]$, moreover $(\eta_i)_{i\geq 1}$ and
$(\zeta_i)_{i\geq 1}$ are independent of both $Z$ and $I$. These requirements can be easily satisfied by extending $\F_0$ and $\F$ if necessary. The sequences $(\eta_i)_{i\geq 1}$ and
$(\zeta_i)_{i\geq 1}$ will be used to construct $\nu_i$ and $(\xi_{j,i})_{j\geq 1}$ in the last paragraph.
As for the filtration $(\cF_t)_{t \in [0,1]}$, we require that $Z/\delta$, as the difference of two independent Poisson
processes with intensity $\beta$, is adapted to $(\cF_t)_{t \in [0,1]}$. Since $Z$ has independent increments and $Z_0=0$, $Z$ is independent of $I$.  We will make one more assumption on the filtration later
during the construction.

Denote by $(\sigma^+_i)_{i\geq 1}$ and $(\sigma^-_j)_{j\geq 1}$ jump times of $Z^B$ and $Z^S$, respectively. We set $\sigma^\pm_i =\infty$ when $\sigma^\pm_i>1$, since we are only interested in processes before time 1. In what follows, we will inductively define two sequences of $[0,1] \cup \{\infty\}$-valued random variables $(\tau^+_i)_{i\geq 1}$ and $(\tau^-_i)_{i\geq 1}$ on $I$. $\tau^+_{i+1}$ (resp. $\tau^-_{i+1}$) will denote the first potential upward (resp. downward) jump of $Y$ after time $\tau_i$ starting with $\tau_0=0$. The process $Y$ on $I$ thus jumps at each $\tau_i := \tau^+_i \wedge \tau^-_i$. In particular, when $\tau^+_i< \tau^-_i$, $\Delta Y_{\tau_i} = \Delta Y^B_{\tau_i} =1$; when $\tau^-_i < \tau^+_i$, $\Delta Y_{\tau_i} = -\Delta Y^S_{\tau_i} = -1$.

Let's start with the construction until the first jump of $Y$. Recall that, in view of Lemma \ref{lem: G intensity}, we want to construct $Y^B$ (resp. $Y^S$) so that its intensity until its first jump is given by
\[
\beta\frac{h(1,t)}{h(0,t)} \quad \left(\mbox{resp. } \beta\frac{h(-1,t)}{h(0,t)}\right).
 \]
Hence $\tau_1$ is constructed to match this intensity.

To define $\tau^+_1$, set
\[
 f_1(t) := 1- \exp\pare{\beta \int_0^t \frac{h(0, u) - h(1, u)}{h(0, u)} du}, \quad t\in [0,1).
\]
Since $z\mapsto h(z, t)$ is strictly increasing, $f_1$ is strictly increasing. We consider the inverse function $f^{-1}_1(y):= \inf\{t\in [0,1) : f(t) > y\}$, where the value is $\infty$ if the indicated set is empty. Now define
\[\nu_1:= f_1^{-1}(\eta_1) \quad \text{ and } \quad \quad \tau^+_1 := \nu_1 \wedge \sigma^+_1 \quad  \text{ on } I.\]
Then $\tau^+_1$ is potentially the first jump time of $Y^B$. It follows from the definition of $\nu_1$ that $\prob(\tau^+_1 <1)>0$. Such $\tau^+_1$ is constructed to match the intensity of $Y^B$ before the first jump of $Y$.
On the other hand, in order to define $\tau^-_1$, consider
\be \label{e:xi}
\xi_{j,1}:= \indic_{\left[\zeta_j \leq \frac{h(-1, \sigma^-_j)}{h(0, \sigma^-_j)}, \, \sigma^-_j<1\right]} + \indic_{[\sigma^-_j\geq 1]} \quad \text{ for } j\geq 1.
\ee
This indicator random variable determines whether the $j^{th}$ jump of $Z^S$ will be cancelled by an opposite jump of $X^{B,S}$.
When $\xi_{j,1}=0$, which only happens when the $j^{th}$ jump of $Z^S$ happens before $1$, this jump of $Z^S$ will be cancelled by a jump of $X^{B,S}$. Such cancelation is performed at a rate $h(-1, \sigma^-_j)/ h(0, \sigma^-_j)$ so as to match the intensity of $Y^S$ before the first jump of $Y$.
Therefore,   $\tau^-_1$,  which is potentially the first negative
jump, is the first jump time $\sigma^-_j$ of $Z^S$ which is {\em not} cancelled. That is,
\[
\tau^-_1:=\min\{\sigma^-_j:\xi_{j,1}=1\}.
\]

Consequently, we define the first jump time of $Y$ on $I$ as
\[
 \tau_1 := \tau^+_1 \wedge \tau^-_1.
\]
This construction yields $\prob(\tau_1 <1)>0$.
On $[t\leq \tau_1, I]$ with $t\leq 1$, we define $X^{B,B}$ and $X^{B,S}$ as
\[
 X^{B,B}_t := \indic_{[\nu_1 < \sigma^+_1]} \indic_{[\tau^+_1 \leq t]} \quad \text{ and } \quad X^{B,S}_t := \sum_{j=1}^\infty (1-\xi_{j,1}) \,\indic_{[\sigma^-_j \leq t]}.
\]

Now suppose that $\tau_{i-1}$ with $\prob(\tau_{i-1}<1)>0$ and $Y_t$ for $t\leq \tau_{i-1}\wedge 1$ have been defined. We will define  in this paragraph $\tau_i$ and $Y_t$ for $t\in (\tau_{i-1}\wedge 1, \tau_i \wedge 1]$. To this end, when $\tau_{i-1}<1$, consider the random function
\[
 f_i(t):= 1-\exp\pare{\lambda\int_{\tau_{i-1}}^t \frac{h(Y_{u-}, u) - h(Y_{u-}+1, u)}{h(Y_{u-}, u)} du}, \quad t\in [\tau_{i-1}, 1).
\]
Since $f_i$ is strictly increasing, the inverse function $f_i^{-1}(y):= \inf\{t\in [\tau_{i-1}, 1) : f(t)>y\}$ is well-defined. When $\tau_{i-1}\geq 1$, set $f_i^{-1}(y)=\infty$.
Now define
\[
 \nu_i := f_i^{-1}(\eta_i) \quad \text{ and } \quad \tau^+_i := \nu_i \wedge \sigma^+_{Z^B_{\tau_{i-1}}+1} \quad \text{ on } I.
\]
To ease notation, we denote $\wt{\sigma}^+_i := \sigma^+_{Z^B_{\tau_{i-1}}+1}$, where $Z^B_{\tau_{i-1}}$ counts the number of $Z^B$ jumps until $\tau_{i-1}$. Hence $\wt{\sigma}^+_i$ indicates which jumps of $Z^B$ could be the next jump of $Y^B$ after $\tau_{i-1}$. Similarly, define
\[
 \xi_{j,i} := \indic_{\left[\zeta_j \leq \frac{h(Y_{\tau_{i-1}}-1, \sigma^-_j)}{h(Y_{\tau_{i-1}}, \sigma^-_j)}, \, \tau_{i-1}\leq \sigma^-_j<1\right]} + \indic_{[\sigma^-_j \geq 1]},
\]
and  set
\[
\tau^-_i:=\min\{\sigma^-_j > \tau_{i-1}: \xi_{j,i} =1\}.
\]
The $i$-th jump of $Y$ on $I$ is then defined as
\[
 \tau_i := \tau^+_i \wedge \tau^-_i.
\]
Since $\prob(\tau_{i-1}<1)>0$, the above construction yields $\prob(\tau_i <1)>0$.
The increment of $X^{B,B}$ and $X^{B,S}$ on $(\tau_{i-1}\wedge 1, \tau_i \wedge 1]$ are defined as
\[
 X^{B,B}_t - X^{B, B}_{\tau_{i-1}} = \indic_{[\wt{\sigma}^+_i> \nu_i]} \indic_{[\tau^+_i \leq t]} \quad \text{ and } \quad X^{B,S}_t - X^{B,S}_{\tau_{i-1}} = \sum_{j=1}^\infty (1-\xi_{j,i})\, \indic_{[\tau_{i-1} \leq \sigma^-_j \leq t]},
\]
for $t\in (\tau_{i-1}\wedge 1, \tau_i \wedge 1]$.

This completes the construction of $X^B$ since $X^B=X^{B,B}+
X^{B,S}$ and we thus obtain the decomposition \eqref{eq: Y decomp} on $I$ for $t \in [0,1 \wedge \lim_{i \rar \infty}\tau_i]$.  As mentioned earlier, the construction on $I^c$ can be performed analogously.

\begin{rem} \label{r:yfinite}
A natural question on whether $\tau:=\lim_{i \rar \infty}\tau_i\geq 1$ or not arises at this point. Observe that since $Z^B$ and $Z^S$ are finite processes, $\prob(\tau<1, I)>0$ implies that there are infinitely many jumps in $X^{B,B}$ so that $\lim_{i \rar \infty}Y^B_{\tau_i}=\infty$, in which case we define $Y^B=\infty$ after $\tau$. A similar explosion on $I^c$ will result in $Y$ becoming $-\infty$. However, we will see in Proposition  \ref{p: Y property} that $Y$ is $\prob$-a.s. a finite process and, thus, $\tau\geq 1$, $\prob$-a.s..
\end{rem}

In order to be able to perform the construction above on $(\Om,\cF,
(\cF_t)_{t\in[0,1]}, \bbP)$, in addition to the assumptions already
imposed on the filtration, we add one more assumption that
$(\F_t)_{t\in[0,1]}$ is right continuous and complete
such that $X^B$ and $X^S$ are $\cF$-adapted and $(\tau_i^+)_{i \geq 1}, (\tau_i^-)_{i \geq 1}$ and $
(\nu_i)_{i\geq 1}$ are $\F$-stopping times. This completes our
assumptions on $(\Omega, \F, (\F_t)_{t\in[0,1]}, \prob)$.
We now return to verify that the process $Y$ just constructed satisfies
\begin{itemize}
\item[i)] $[Y_1\geq y_1]=I$, $\prob$-a.s., and
\item[ii)] In its own filtration $Y=Y^B-Y^S$ where $Y^B$ and $Y^S$ are $\cF^Y$-adapted independent Poisson processes with intensity $\beta$.
\end{itemize}
We first establish that the $\F$-intensity of $Y$ is of the same form as the $\cG^1$-intensity of $Z$ computed in Lemma \ref{lem: G intensity}.

\begin{lem}\label{lem: F intensity}
 The $\F$-intensities of $Y^B$ and $Y^S$ at $t\in[0,1)$ are given by
 \[
  \indic_{I} \beta \frac{h(Y_{t-} +1, t)}{h(Y_{t-}, t)} + \indic_{I^c} \beta \frac{1-h(Y_{t-}+1, t)}{1-h(Y_{t-}, t)} \quad \text{ and } \quad
  \indic_{I} \beta \frac{h(Y_{t-} -1, t)}{h(Y_{t-}, t)} + \indic_{I^c} \beta \frac{1-h(Y_{t-}-1, t)}{1-h(Y_{t-}, t)}.
 \]
\end{lem}

\begin{proof}
 We will calculate the $\F$-intensities of $Y^B$  and $Y^S$ on $I$. Their intensities on $I^c$ can be similarly verified. First, observe that the construction of $\nu_i$ implies that on $I$ and $[\tau_{i-1} <1]$
 \begin{equation}\label{e:pnu}
 \begin{split}
 & \prob\pare{\nu_i > t\vee \tau_{i-1} \,|\, \F_{\tau_{i-1}}} \\
 &\quad = \prob\pare{\eta_i > f_i(t \vee \tau_{i-1})\,|\, \F_{\tau_{i-1}}} \\
 &\quad = \exp\pare{\beta \int_{\tau_{i-1}}^{t\vee \tau_{i-1}} \frac{h(Y_{u-}, u) - h(Y_{u-}+1, u)}{h(Y_{u-}, u)} du}, \quad \text{ for } t\in[0,1).
 \end{split}
 \end{equation}
 We will make repeated use of (\ref{e:pnu}) in order to obtain the $\F$-intensity of $Y^B$ on $I$. To this end, note that $[\tau^+_i> t\geq \tau_{i-1}, I] = [\wt{\sigma}^+_i > t, \nu_i> t, t\geq \tau_{i-1}, I]$. Therefore we have on $[t\geq \tau_{i-1}, I]$ that
 \be  \label{eq: dist tau_i}
 \begin{split}
  \prob(\tau^+_i > t\,|\, \F_{\tau_{i-1}})&= \prob(\wt{\sigma}^+_i > t \,|\, \F_{\tau_{i-1}}) \prob(\nu_i > t \,|\, \F_{\tau_{i-1}}) \\
  &= \prob(Z_t= Z_{\tau_{i-1}} \,|\, \F_{\tau_{i-1}}) \exp\pare{\beta \int_{\tau_{i-1}}^{t} \frac{h(Y_{u-}, u) - h(Y_{u-}+1, u)}{h(Y_{u-}, u)} du}\\
  &= \exp(-\beta (t-\tau_{i-1}))\exp\pare{\beta \int_{\tau_{i-1}}^{t} \frac{h(Y_{u-}, u) - h(Y_{u-}+1, u)}{h(Y_{u-}, u)} du},
 \end{split}
 \ee
where the first line is due to the independence of $Z^B$ and $\nu_i$ and the last line follows from the strong Markov property of $Z^B$ and the fact that $\tau_{i-1}$ is an $\cF$-stopping time.

It is well-known (see, e.g. Proposition 3.1 in \cite{Jacod}) that the  $\F$-intensity of $Y^B$ on $I$ is given by
 \[
 \frac{\prob(\tau^+_i \in dt \,|\, \F_{\tau_{i-1}})}{\prob(\tau^+_i>t \,|\, \F_{\tau_{i-1}})dt}, \quad t\in(\tau_{i-1}, \tau_i].
 \]
 Utilising \eqref{eq: eqn h} and \eqref{eq: dist tau_i}, it follows from direct calculations and the observation that $Y_{t-}$ is constant in $(\tau_{i-1}, \tau_i]$ that the above intensity is indeed
 \[
  \beta \frac{h(Y_{t-}+1, t)}{h(Y_{t-}, t)} \quad \text{ on } I \text{ for } t\in (\tau_{i-1}, \tau_i].
 \]

 To calculate the intensity of $Y^S$ on $I$ and in the time interval
 $(\tau_{i-1}, \tau_i]$, we will treat the evolution of $Y^S$ as that
 of a marked point process with points $(\sigma^-_j \vee
 \tau_{i-1})_{j\geq 1}$ and marks $(\xi_{j,i})_{j\geq 1}$. Let
 $\wt{\sigma}^-_i:=\sigma^-_{Z^S_{\tau_{i-1}}+1} $ and $\wt{\xi}_i$ be the associated mark. Define $G_i(dt,
 1) = \prob\left(\wt{\sigma}^-_i \in dt, \wt{\xi}_i = 1 \,|\, \F_{\tau_{i-1}}\right)$ and $H_i(dt) = \prob(\wt{\sigma}^-_i\in dt \,|\, \F_{\tau_{i-1}})$. It then follows from Proposition 3.1 in \cite{Jacod} that the intensity of $Y^S$ at $t\in (\tau_{i-1}, \tau_i]$ is
 \[
 \begin{split}
  \frac{G_i(dt, 1)}{H_i([t, \infty]) dt} & =\frac{E_\prob\bra{\indic_{[\wt{\sigma}^-_i \in dt]}\prob(\wt{\xi}_i=1
      \,|\, \F_{\wt{\sigma}^-_i })\,\big|\, \F_{\tau_{i-1}}}}{\prob(\wt{\sigma}^-_i > t \,|\, \F_{\tau_{i-1}}) dt}\\
  &= \frac{h(Y_{t-} -1, t)}{h(Y_{t-}, t)} \frac{\prob(\wt{\sigma}^-_i\in dt \,|\, \F_{\tau_{i-1}})}{\prob(\wt{\sigma}^-_i > t\,|\, \F_{\tau_{i-1}}) dt} \\
  &= \frac{h(Y_{t-} -1, t)}{h(Y_{t-}, t)} \beta,
 \end{split}
 \]
due to the strong Markov property of $Z^S$. This verifies the intensity of $Y^S$ on $I$.
\end{proof}
We are now ready to prove that our construction as desired.
\begin{prop}\label{p: Y property}
 The process $(Y_t; t\in[0,1])$ as constructed above satisfies the following properties:
 \begin{enumerate}
  \item[i)] $[Y_1 \geq y_1] = I$, $\prob$-a.s.;
  \item[ii)] $Y^B$ and $Y^S$ are independent Poisson processes with intensity $\beta$ with
    respect to the natural filtration $(\F^Y_t)_{t\in[0,1]}$ of $Y$. In particular, $Y$ is finite $\prob$-a.s. over $[0,1]$.
  \item[iii)] $\expec[X^B_1]$ and $\expec[X^S_1]$ are finite. Hence the constructed strategy $(X^B, X^S; \F^I)$ is admissible.
 \end{enumerate}
\end{prop}

\begin{proof}
 To verify that $Y$ satisfies the desired properties, let us introduce an auxiliary process $(\ell_t)_{t\in[0,1)}$ via
\[
 \ell_t := \indic_I \frac{h(0,0)}{h(Y_t, t)} + \indic_{I^c} \frac{1-h(0,0)}{1-h(Y_t, t)}, \quad t\in [0,1).
\]
The construction of $Y^S$ on $I$ (resp. $Y^B$ on $I^c$) implies that
there are only a finite number of jumps before a fixed time
$t<1$. Therefore $Y_t > -\infty$ on $I$ (resp. $Y_t<\infty$ on $I^c$)
for $t\in [0,1)$, which implies $h(Y_t, t)>0$ on $I$ (resp. $h(Y_t,
t)< 1$ on $I^c$) for $t\in [0,1)$. As a result, $(\ell_t)_{t\in [0,1)}$ is a well-defined positive process with $\ell_0=1$. To prove the first statement, we first show that $\ell$ is a positive $\F$-local martingale on $[0, 1)$. To this end, Ito's formula yields that
 \[
 \begin{split}
  d\ell_t =\indic_I \ell_{t-} &\bra{\frac{h(Y_{t-}, t)- h(Y_{t-}+1, t)}{h(Y_{t-}+1, t)} dM^{B}_t + \frac{h(Y_{t-}, t) - h(Y_{t-}-1, t)}{h(Y_{t-}-1, t)} dM^{S}_t}\\
    +\indic_{I^c} \ell_{t-} &\bra{\frac{h(Y_{t-}+1, t)- h(Y_{t-}, t)}{1-h(Y_{t-}+1, t)} dM^{B, c}_t + \frac{h(Y_{t-}-1, t) - h(Y_{t-}, t)}{1-h(Y_{t-}-1, t)} dM^{S,c}_t}, \quad t\in[0,1).
 \end{split}
 \]
 Here
 \begin{align*}
   &M^B = \indic_I  Y^B - \indic_I \beta \int_0^\cdot \frac{h(Y_{u-} +1, u)}{h(Y_{u-}, u)} du, &M^S = \indic_I  Y^S - \indic_I \beta \int_0^\cdot \frac{h(Y_{u-}-1, u)}{h(Y_{u-}, u)} du,\\
   &M^{B, c} =\indic_{I^c}  Y^B - \indic_{I^c} \beta \int_0^\cdot \frac{1-h(Y_{u-} +1, u)}{1-h(Y_{u-}, u)} du, &M^{S, c} =\indic_{I^c}  Y^S - \indic_{I^c} \beta \int_0^\cdot \frac{1- h(Y_{u-} -1, u)}{1-h(Y_{u-}, u)} du
  \end{align*}
  are all $\F$-local martingales. Define $\zeta^+_n = \inf\{t\in[0,1]: Y_t = n\}$ and $\zeta^-_n = \inf\{t\in[0,1] : Y_t = -n\}$. Consider the sequence of stopping times $(\eta_n)_{n\geq 1}$, where
  \[
   \eta_n := \pare{\indic_I \zeta^-_n + \indic_{I^c} \zeta^+_n} \wedge (1-1/n).
  \]
  It follows from the definition of $h$  that $h(Y_t, t)$ on $I$ (resp. $1-h(Y_t,t)$ on $I^c$) is bounded away from zero uniformly in $t\in[0, \eta_n]$. This implies that $\ell^{\eta_n}$ is bounded, hence $\ell^{\eta_n}$ is a $\F$-martingale. The construction of $Y^S$ on $I$ (resp. $Y^B$ on $I^c$) yields  $\lim_{n\rightarrow \infty} \eta_n = 1$. Therefore, $\ell$ is a positive $\F$-local martingale, hence also a supermartingale, on $[0,1)$.

  Define $\ell_1:= \lim_{t\rightarrow 1} \ell_t$, which exists and is
  finite due to Doob's supermartingale convergence theorem. This implies $h(Y_{1-}, 1)>0$ on $I$ (resp. $1-h(Y_{1-}, 1)>0$ on $I^c$). Recall that $Y^S$ on $I$ (resp. $Y^B$ on $I^c$) does not jump at time $1$ almost surely. Therefore $h(Y_{1}, 1)>0$ on $I$ (resp. $1-h(Y_{1}, 1)>0$ on $I^c$), which yields $Y_1 \geq y_1$ on $I$ (resp. $Y_1 < y_1$ on $I^c$).

  Let us now prove the second statement for $Y^B$. The statement for $Y^S$ can be shown similarly.
  In view of the $\F$-intensity of $Y^B$ calculated in Lemma \ref{lem:
    F intensity}, one has that, for each $i \geq 1$
  \[
   Y^B_{\cdot \wedge \tau_i \wedge 1} -\beta \pare{\indic_I \int_0^{\cdot\wedge \tau_i \wedge 1} \frac{h(Y_{u-}+1, u)}{h(Y_{u-}, u)} du+ \indic_{I^c} \int_0^{\cdot\wedge \tau_i \wedge 1} \frac{1-h(Y_{u-}+1, u)}{1-h(Y_{u-}, u)}du}
  \]
  is an $\F$-martingale. We will show in the next paragraph that, when stopped at $\tau_i\wedge 1$, $Y^B$ is a Poisson process in $\F^Y$ by  showing that  $(Y^B_{t\wedge \tau_i} - \beta (t\wedge \tau_i))_{t\in [0,1]}$ is a $\F^Y$-martingale--recall that $\tau_i$ is an $\cF^Y$-stopping time. This in turn will imply that $Y^B$ is a Poisson process with intensity $\beta$ on $[0,\tau \wedge 1)$ where $\tau=\lim_{i \rar \infty} \tau_i$ is the explosion time. Since Poisson process does not explode, this will further imply $Y^B_{\tau \wedge 1}< \infty$ and, therefore, $\tau\geq 1,\, \prob$-a.s. in view of Remark \ref{r:yfinite}.

  We proceed by projecting the above martingale into $\F^Y$ to see that
  \[
   Y^B- \beta \int_0^\cdot \bra{\prob(I\,|\, \F^Y_u) \frac{h(Y_{u-}+1, u)}{h(Y_{u-}, u)} + \prob(I^c\,|\, \F^Y_u) \frac{1-h(Y_{u-}+1, u)}{1-h(Y_{u-}, u)}}du
  \]
  is a $\F^Y$-martingale when stopped at $\tau_i\wedge 1$. Therefore, it remains to show that, for
  almost all  $t\in [0,1)$, on $[t \leq \tau_i]$
  \[
  \prob(I\,|\, \F^Y_t) \frac{h(Y_{t-}+1, t)}{h(Y_{t-}, t)} + \prob(I^c\,|\, \F^Y_t) \frac{1-h(Y_{t-}+1, t)}{1-h(Y_{t-}, t)} =1.
  \]
  In the remaining of the proof, we will show that on $[t \leq \tau_i]$
  \begin{equation}\label{eq: prob I}
   \prob(I \,|\, \F^Y_t) = h(Y_{t}, t) \quad \text{ and } \quad \prob(I^c\,|\, \F^Y_t) = 1- h(Y_{t}, t), \quad \text{ for } t\in[0,1).
  \end{equation}
  The statement then follows since $Y_t \neq Y_{t-}$ only for
  countably many times.

 We have seen  that $(\ell_{u\wedge \tau_i})_{u\in[0,t]}$ is a
  strictly positive $\F$-martingale for each $i$. Define a
  probability measure $\qprob^i \sim \prob$ on $\F_t$ via
  $d \qprob^i/d \prob |_{\F_t} = \ell_{\tau_i \wedge t}$. It
  follows from a simple application of Girsanov's theorem that
  $(Y^B_{\cdot})$ and $(Y^S_{\cdot})$ are
  Poisson processes when stopped at $\tau_i \wedge t$ and  with
  intensity $\beta$ under $\qprob^i$. Therefore, they are independent
  from $I$ under $\qprob^i$.   Then, for $t< 1$ we obtain from Bayes' formula that
  \begin{equation}\label{eq: p=h}
  \begin{split}
   \indic_{[u\leq \tau_i \wedge t]} \prob(I \,|\, \F^Y_u) & = \indic_{\{u\leq \tau_i \wedge t\}} \frac{E_{\qprob^i}\bra{ \indic_I \ell^{-1}_{u} \,|\, \F^Y_u}}{E_{\qprob^i}\bra{\ell^{-1}_{u}\,|\, \F^Y_u}}\\
   &=  \indic_{[u\leq \tau_i\wedge t]} \frac{E_{\qprob^i} \bra{\indic_I \frac{h(Y_u, u)}{h(0,0)} \,|\, \F^Y_u}}{E_{\qprob^i} \bra{\indic_I \frac{h(Y_u, u)}{h(0,0)} + \indic_{I^c} \frac{1-h(Y_u, u)}{1-h(0,0)} \,|\, \F^Y_u}}\\
   &= \indic_{[u\leq \tau_i \wedge t]} h(Y_u, u),
  \end{split}
  \end{equation}
  where the third identity follows from the aforementioned
  independence of $Y$ and $I$ under $\qprob^i$ along with the fact
  that $\qprob^i$ does not change the probability of $\F_0$ measurable
  events, so that $\qprob^i(I) = \prob(I) = h(0,0)$. As a result, \eqref{eq: prob I} follows from \eqref{eq: p=h} after sending $i\rightarrow \infty$.

  Finally, since $Y^B$ and $Y^S$ are Poisson processes in $\F^Y$ and they do not jump simultaneously by their construction, they are independent (see \cite[Proposition 5.3]{Cont-Tankov}). Since $Y^B$ and $Y^S$ are independent Poisson processes, it also follows $\expec[X^B_1]< \infty$. Indeed, since $X^B_1 \indic_{I^c}=0$, we have $\expec[X^B_1] =\expec[X^B_1 \indic_I]=\expec[(Y^B_1-Z^B_1+X^{B,S})\indic_I]\leq \expec[Y^B_1]+ \expec[Z^B_1] + \expec[Z^S_1] <\infty$. Similar arguments also show that $\expec[X^S_1]<\infty$. Hence, the constructed strategy $(X^B, X^S; \F^I)$ is admissible.
\end{proof}

\section{Existence and convergence of Glosten-Milgrom equilibria}\label{sec: equilibrium existence}
In view of the results of  Section \ref{sec: bridge construction}, we can now show that a Glosten-Milgrom equilibrium exists for the market model under consideration.

\begin{thm}\label{thm: existence equilibrium}
Suppose that $(\eta_i)_{i \geq 1}$ and $(\zeta_i)_{i \geq 1}$ are two sequences of independent $\F$-measurable random variables uniformly distributed over $[0,1]$ that are independent from each other, $Z$ and $\tv$. If there exists a $y_{\delta}$ such that
\[
\bbP(Z_1 \geq y_{\delta})=\bbP(\tv=1),
\]
and $\cF^I$ is the right continuous augmentation of $(\sigma(\tv, Z_s, \eta_i, \zeta_i; s\leq t, i \geq 1))_{t \in [0,1]}$ with the $\bbP$-null sets, then there exists a Glosten-Milgrom equilibrium.
\end{thm}
\begin{proof} In view of Theorem \ref{t:equilibrium}, an equilibrium
  exists if $Y$ satisfies the conditions stated in Theorem
  \ref{t:equilibrium} and the high type (resp. low type) insider never
  sells (resp. buys). However, the insider can use the uniform random
  variables available in her filtration to perform the construction
  described in Section \ref{sec: bridge construction} so that $Y$
  satisfies the desired properties, due to Proposition \ref{p: Y
    property}, without having to sell (resp. buy) when high type
  (resp. low type).
\end{proof}

In the remainder of this section we will analyse what happens when the
trade size becomes small $(\delta \rightarrow 0)$ and the noise trades
arrive more frequently $(\beta \rightarrow \infty)$. A similar convergence
has also been studied by Back and Baruch in \cite{Back-Baruch} who
have established that the limiting economy can be described by a Kyle-Back
equilibrium. We would like to mention at this point that Back and
Baruch have proved their convergence results under some extra
hypotheses on the convergence of value functions which may be hard to verify. As we shall see below, we will verify the convergence via a weak convergence approach and we do not need any extra assumptions
in addition to the ones which have  already been assumed. Before performing a weak convergence analysis of
Glosten-Milgrom equilibria, whose existence is justified by  Theorem
\ref{thm: existence equilibrium}, let's first briefly describe what we
mean by a Kyle-Back equilibrium.

The continuous-time model of Kyle \cite{Kyle}, which was later
extended by Back \cite{Back}, studies the equilibrium pricing of a
risky asset whose liquidation value at time $1$ is given by
$\tilde{v}$. In this model, the cumulative noise trades is modelled by
a Brownian motion, denoted with $W$, independent of $\tilde{v}$. The
risk neutral insider knows the true liquidation value from the
beginning and competition among the risk neutral market makers forces
them to quote prices as conditional expectations of $\tilde{v}$ based
on their information. The price is again set in a Markovian manner,
i.e. there exists a function $p^0:\bbR\mapsto [0,1]$ so that the
market price is given by $p^0(Y_t,t)$ at time $t$ where $Y$ is, as
before, the cumulative demand at time $t$.

Let $\Omega^0 = \mathbb{D}([0,1], \Real)$ be the space of
$\Real$-valued \cadlag \, functions on $[0,1]$ with the coordinate
process $Y^0$ and $\bbP^0$ be the Wiener measure. In view of the results of
\cite{Back} and \cite{CCD}, the  equilibrium price of the risky asset
in this economy is given by
\be \label{e:kbp}
 p^{0}(y, t) := \prob^0_y\bra{Y^0_{1-t}\geq y_0},
\ee
where
\[
 y_0 := \Phi^{-1}(1-{\prob}(\tv=1)),
\]
and $\Phi(\cdot) = \int_{-\infty}^\cdot \frac{1}{\sqrt{2\pi}}
e^{-x^2/2} dx$.
The equilibrium demand satisfies the SDE
\be \label{e:kbY}
Y= W +\indic_{[\tilde{v}=1]}\int_0^\cdot{\partial_y}
\log p^0(Y_s,s)\,ds+\indic_{[\tilde{v}=0]}\int_0^\cdot{\partial_ y}\log(1- p^0(Y_s,s))\,ds.
\ee
\begin{rem} \label{r:kbe} Strictly speaking, the equilibrium price (\ref{e:kbp})
  and demand (\ref{e:kbY}) in this economy do not follow directly from the
  results of \cite{Back} and \cite{CCD} since in their framework
  $\tilde{v}$ has a continuous distribution. However, if one follows the
  arguments for the description of equilibrium given in \cite{CCD}, it
  follows that in equilibrium the insider trades so that
  $\tilde{v}=p^0(Y_1,1)$ and $Y$ is a Brownian motion in its own
  filtration. This immediately gives  (\ref{e:kbp}) as the equilibrium
  price, since the price follows a martingale with respect to the
  filtration of the market maker, which is the same as the filtration
  generated by $Y$. Moreover,  the same characterisation gives that
  the SDE satisfied by $Y$ with respect to the filtration of the
  insider is the same as the SDE satisfied by a standard Brownian
  motion when its natural filtration is initially
  enlarged with the random variable corresponding to its time $1$
  value being larger than $y_0$. The standard arguments contained in,
  e.g. Section 1.3 of \cite{Mansuy-Yor}, gives (\ref{e:kbY}).
\end{rem}
In view of the well-known results on the weak convergence of a
sequence of difference of Poisson processes to Brownian motion (see,
e.g., Theorem 5.4 in Chapter~6 of \cite{Ethier-Kurtz}), it is
easy to see that the cumulative demand of noise traders in a Kyle-Back
model can be considered as the weak limit of noise demands in a
sequence of Glosten-Milgrom models. Based on this observation it is
natural to ask whether the Kyle-Back equilibrium is the weak limit of
Glosten-Milgrom equilibria.

We now return to give an affirmative answer to this question. More
precisely, we consider the convergence of Glosten-Milgrom equilibria to the
Kyle-Back equilibrium described by (\ref{e:kbp}) and (\ref{e:kbY}). In
what follows, the superscript $\delta\geq 0$ indicates the trade size
associated to different processes, probabilities, random variables,
and functions. 

Let $(\Omega^\delta, \F^\delta, (\F^\delta_t)_{t\in[0,1]}, \prob^\delta)_{\delta \geq 0}$ be a sequence of probability spaces on which the Glosten-Milgrom models of different order sizes are defined. When $\delta >0$,
$\Omega^\delta = \mathbb{D}([0,1], \delta \Integer)$ is the space of
$\delta\Integer$-valued \cadlag\, functions on $[0,1]$ with the
coordinate process $Y^\delta$, $(\F^\delta_t)_{t\in[0,1]}$ is the
minimal right continuous and complete filtration generated by
$Y^\delta$, and $\prob^\delta$ for $\delta >0$ is the probability
measure under which $Y^\delta$ is the difference of two independent
Poisson processes with the same intensity $\beta^\delta$. $\bbP^0$ is
the Wiener measure as mentioned in the earlier paragraphs.

To construct a sequence of pricing rules in Glosten-Milgrom equilibria which converges to the Kyle-Back equilibrium, set
\[
 y_\delta := \inf\{y\in \delta \Integer, \, \prob^\delta(Y^\delta_1 \leq y) \geq 1-{\prob}(\tv=1)\}, \quad \text{ for } \delta >0,
\]
and denote $p^{y_\delta}$, defined in \eqref{eq: pz}, by $p^\delta$ for simplicity. To ensure the existence of Glosten-Milgrom equilibria with pricing rules $(p^\delta)_{\delta >0}$, we introduce a sequence of Bernoulli random variables $(\tv^\delta)_{\delta>0}$ whose distribution is
\begin{equation}\label{eq: dist tv delta}
{\prob}(\tv^\delta =1) = \prob^\delta(Y^\delta_1 \geq y_\delta).
\end{equation}
These $(\tv^\delta)_{\delta>0}$ will be the liquidation values of the risky asset in the sequence of Glosten-Milgrom models which converges to the Kyle-Back model.

\begin{thm}\label{thm: convergence}
For any $\tilde{v}$ satisfying  ${\prob}(\tv =1) \in (0,1)$, there exists a sequence of admissible strategies $(X^{B, \delta}, X^{S, \delta})_{\delta>0}$ such that, for each $\delta>0$, $(p^\delta, X^{B, \delta}, X^{S,\delta})$ is a Glosten-Milgrom equilibrium whose fundamental value of the risky asset is $\tv^\delta$.

 When the intensity of Poisson process is given by $\beta^\delta = (2\delta^2)^{-1}$ in the Glosten-Milgrom model, as $\delta \rightarrow 0$, the sequence of Glosten-Milgrom equilibria converge to the Kyle-Back equilibrium in the following sense:
 \begin{enumerate}
 \item[i)] The bid and ask prices in these Glosten-Milgrom equilibria
   converge to the price in the Kyle-Back equilibrium. That is,
   $\lim_{\delta\downarrow 0} a^{\delta}(y, t) =
   \lim_{\delta\downarrow 0} b^{\delta}(y, t)= \lim_{\delta\downarrow
     0} p^{\delta}(y, t) = p^{0}(y, t)$ for $(y, t) \in \Real \times
   [0,1)$.
Moreover, the corresponding market depths in the Glosten-Milgrom
equilibria converges to the market depth in the Kyle-Back equilibrium:
     \begin{equation*}
  \lim_{\delta\downarrow 0} \frac{1}{\delta} \pare{a^{\delta}(y, t) - p^{\delta}(y, t)} =\lim_{\delta \downarrow 0} \frac{1}{\delta} \pare{p^{\delta}(y, t) - b^{\delta}(y,t)} = \partial_y p^{0}(y,t), \quad \text{ for } (y,t)\in \Real \times [0,1).
 \end{equation*}
 \item[ii)]  Let $Y^{0, H}$ and $Y^{0, L}$ be the solutions to the following two SDEs, respectively,
     \[
      dY_t = \frac{\partial_y p^{0}(Y_t, t)}{p^{0}(Y_t, t)} \,dt + dW_t \quad \text{ and } \quad
      dY_t = -\frac{\partial_y p^{0}(Y_t, t)}{1-p^{0}(Y_t, t)} \,dt +
      dW_t, \qquad t \in [0,1),
     \]
     where $W$ is a Brownian motion under $(\Omega, \F^0, (\F^0_t)_{t\in[0,1]}, \prob^0)$. Define
     \[
      B^{0}_\cdot = \int_0^{\cdot} \frac{\partial_y p^{0}(Y^{0,H}_t, t)}{p^{0}(Y^{0,H}_t, t)} \,dt \quad \text{ and } \quad
      S^{0}_\cdot = \int_0^{\cdot} \frac{\partial_y p^{0}(Y^{0,L}_t, t)}{1-p^{0}(Y^{0,L}_t, t)} \,dt.
     \]
     Then,
     \begin{itemize}
      \item When $\tv=1$, $X^{B, \delta} \overset{\mathscr{L}}{\rightarrow} B^0$;
      \item When $\tv=0$, $X^{S, \delta} \overset{\mathscr{L}}{\rightarrow} S^0$,
     \end{itemize}
     where $\overset{\mathscr{L}}{\rightarrow}$ represents the convergence in law.

 \item[iii)]  $(p^{0}, Y^0)$ satisfies (\ref{e:kbp}) and (\ref{e:kbY})
   where
\[
Y^0=\indic_{[v=1]}Y^{0,H}+ \indic_{[v=0]}Y^{0,L}.
\]
As such, $p^0$ and $Y^0$ are the equilibrium price and demand in the
Kyle-Back equilibrium, respectively.
 \end{enumerate}
\end{thm}

The above theorem tells us that Kyle-Back model with Bernoulli distributed $\tv$ can be approximated by a sequence of Glosten-Milgrom models whose risky asset fundamental price converges to $\tv$ in distribution. Since there
is no bid-ask spread in the Kyle-Back equilibrium, the above
convergence results in particular tell us that the bid-ask spread gets
smaller and vanish in the limit as the frequency of noise trades
increase. Moreover the rate the convergence is $O(\delta)$.

To show the desired convergence results contained in the theorem above, let us first prove the
convergence in law of the cumulative order processes as seen in the
filtration, say $\overline{\F}^{\delta}$,
with respect to which $X^\delta$ and $Z^\delta$ are
adapted and $\tv \in \overline{\F}^{\delta}_0$. Note that this filtration is smaller than the filtration that
is assumed to be contained in the insider's filtration in Theorem
\ref{thm: existence equilibrium}, however, it contains all the
relevant processes and random variables describing insider's strategy
and the informational advantage. Moreover, it will be enough to
limit ourselves to these filtrations in order to prove Theorem
\ref{thm: convergence}. Recall from Section \ref{sec: bridge
  construction} that,
for each $\delta>0$, the distribution of the cumulative order process in $\overline{\F}^{\delta}$ is the same as the distribution of $Y^\delta$ conditioned on $Y^\delta_1 \geq y_\delta$ or $Y^\delta_1 < y_\delta$. Here $Y^\delta/\delta$ is the difference of two Poisson processes in its own filtration.

\begin{lem}\label{lem: cumulative order conv}
 Let $\beta^\delta = (2\delta^2)^{-1}$. We have
 \[
  \text{Law}(Y^\delta \,|\, Y^\delta_1 \geq y_\delta) \Rightarrow \text{Law}(Y^{0,H}) \quad \text{ and } \quad  \text{Law}(Y^\delta \,|\, Y^\delta_1 < y_\delta) \Rightarrow \text{Law}(Y^{0,L}), \quad \text{ as } \delta \rightarrow 0,
 \]
 where $Y^{0,H}$ and $Y^{0,L}$ are defined in Theorem \ref{thm: convergence} ii) and $\Rightarrow$ represents the weak convergence of probability measures.
\end{lem}

\begin{proof}
 The first convergence will be proved. The second convergence can be
 shown similarly. Since $\beta^\delta = (2\delta^2)^{-1}$, it follows
 from \cite[Theorem 5.4 in Chapter~6]{Ethier-Kurtz} that $\prob^\delta
 \Rightarrow \prob^0$ and, in particular, $Law(Y^\delta_1) \Rightarrow
 Law(Y^0_1)$. Observe that $y_\delta$ is the $(1-\prob(\tv=1))^{th}$ quantile of the distribution for $Y^\delta_1$ and the distribution of $Y^0_1$ is continuous. It then follows
 \begin{equation}\label{eq: y delta conv}
  \lim_{\delta \downarrow 0} y_\delta = y_0.
 \end{equation}

 Meanwhile the conditional distribution $Law(Y^\delta \,|\, Y^\delta_1 \geq y_\delta)$ is defined via
 \begin{equation}\label{eq: conditional law}
  \prob^{\delta, H}(A) := \frac{\prob^\delta(A, Y^\delta_1\geq y_\delta)}{\prob^\delta(Y^\delta_1 \geq y_\delta)}, \quad \text{ for } A\in \F^\delta.
 \end{equation}

We will show $\prob^{\delta, H} \Rightarrow \prob^{0,H}$ as $\delta
\downarrow 0$. This statement will follow once we show the finite dimensional
distributions of $Y^{\delta}$ converge weakly  to the finite
dimensional distributions of $Y^0$, and $(\prob^{\delta,
  H})_{\delta>0}$ is tight (see e.g. \cite[VI.3.20]{Jacod-Shiryaev}). We will prove both of these conditions using
the already observed convergence of $\bbP^{\delta}$ to $\bbP^0$.

To this end, we will first establish the the convergence of
$\prob^\delta(Y^\delta_1 \geq y_\delta)$ to $\prob^0(Y^0_1 \geq y_0)$.
Indeed, due to \eqref{eq: y delta conv}, there exists a sufficiently small $\delta_\epsilon$ such that $y_\delta \geq y_0 -\epsilon$ for $\delta \leq \delta_\epsilon$. Thus,
 \[
  \prob^\delta(Y^\delta_1 \geq y_\delta) \leq \prob^\delta(Y^\delta_1 \geq y_0-\epsilon) \rightarrow \prob^0(Y^0_1 \geq y_0-\epsilon), \quad \text{ as } \delta \downarrow 0,
 \]
 where the convergence follows from $Law(Y^\delta_1) \Rightarrow Law(Y^0_1)$ and the fact that the distribution of $Y^0_1$ is continuous at $y_0-\epsilon$. Then the previous inequality yields $\limsup_{\delta \downarrow 0} \prob^\delta(Y_1^\delta \geq y_\delta) \leq \prob^0(Y^0_1 \geq y_0)$ since the choice of $\epsilon$ is arbitrary. Combining the previous inequality with $\liminf_{\delta \downarrow 0} \prob^\delta(Y^\delta_1 \geq y_\delta) \geq \prob^0(Y^0_1\geq y_0)$, which can be similarly proved, we obtain
 \begin{equation}\label{eq: h delta conv}
  \lim_{\delta \downarrow 0} \prob^\delta(Y^\delta_1 \geq y_\delta) = \prob^0(Y^0_1 \geq y_0) = \prob(\tv=1)>0.
 \end{equation}

 In order to prove the convergence of the finite dimensional
distributions of $Y^{\delta}$,  we are first going to show
 \[
  \lim_{\delta \downarrow 0} \expec^{\prob^{\delta, H}}\bra{f(Y^\delta_{t_1}, \cdots, Y^\delta_{t_n})} = \expec^{\prob^{0, H}}\bra{f(Y^0_{t_1}, \cdots, Y^0_{t_n})},
 \]
 for arbitrary bounded continuous function $f: \Real^n \rightarrow
 \Real$ and $0\leq t_1\leq \cdots < t_n\leq 1$. However, similar
 arguments as those employed in the last paragraph yield
 \[
 \lim_{\delta \downarrow 0} \expec^{\prob^{\delta}}\bra{f(Y^\delta_{t_1}, \cdots, Y^\delta_{t_n}) \,\indic_{[Y^\delta_1 \geq y_\delta]}} = \expec^{\prob^{0}}\bra{f(Y^0_{t_1}, \cdots, Y^0_{t_n}) \,\indic_{[Y^0_1\geq \delta_0]}}.
 \]
 The claim then follows from combining the previous convergence with \eqref{eq: conditional law} and \eqref{eq: h delta conv}.

 To verify the tightness of $(\prob^{\delta, H})_{\delta>0}$, it is equivalent to prove the following two conditions (see \cite[Theorem VI.3.21]{Jacod-Shiryaev}):
 \begin{enumerate}
  \item for any $\epsilon>0$, there exist $\delta_\epsilon$ and $K\in \Real$ with
      \[
       \prob^{\delta, H}\pare{\sup_{0\leq t\leq 1} |Y^\delta_t| >K} \leq \epsilon, \quad \text{ for all } \delta \leq \delta_\epsilon;
      \]
  \item for any $\epsilon>0$ and $\eta>0$, there exists $\delta_{\epsilon, \eta}$ and $\theta_{\epsilon, \eta}$ such that
      \[
       \prob^{\delta, H}\pare{w'_1(Y^\delta, \theta_{\epsilon, \eta}) \geq \eta} \leq \epsilon, \quad \text{ for all } \delta \leq \delta_{\epsilon, \eta}.
      \]
      We refer reader to \cite[Chapter VI, Section 1a]{Jacod-Shiryaev} for the definition of $w'_1$.
 \end{enumerate}
 Observe that, since $\prob^\delta \Rightarrow \prob^0$,
 $(\prob^\delta)_{\delta>0}$ is tight which implies that  the
 two conditions above hold when $\prob^{\delta, H}$ is replaced by
 $\prob^{\delta}$.  Moreover, if  $A$ stands for the event
 $[\sup_{0\leq t\leq 1} |Y^\delta_t|>K]$ or $[w'_1(Y^\delta,
 \theta_{\epsilon, \eta}) \geq \eta]$, \eqref{eq: conditional law}
 along with \eqref{eq: h delta conv}  yields
 \[
  \prob^{\delta, H}(A) = \frac{\prob^{\delta}(A, Y^\delta_1\geq
    y_\delta)}{\prob^{\delta}(Y^\delta_1 \geq y_\delta)}\leq  \frac{\prob^{\delta}(A)}{\prob^{\delta}(Y^\delta_1 \geq y_\delta)} \leq \frac{\epsilon}{\prob^{\delta}(Y^\delta_1 \geq y_\delta)} \leq \frac{2\epsilon}{\prob^0(Y^0_1 \geq y_0)} \quad \text{ for sufficiently small } \delta,
 \]
 which confirms the aforementioned conditions for $\prob^{\delta, H}$.

 Finally, it remains to verify that $\prob^{0, H}$ is the law of
 $Y^{0,H}$. To this end, note that $\prob^{0,H}$ is the law of a
 Brownian motion conditioned on its time 1 value being larger than
 $y_0$. A standard calculation using the well-known h-transform
 technique gives the following semimartingale decomposition of $Y^0$
under $\bbP^{0,H}$:
 \[
  Y^0_t = \int_0^t \frac{\partial_y p^0(Y^0_u, u)}{p^0(Y^0_u, u)} \, du + \widetilde{W}_t, \quad t\in[0,1),
 \]
 where $\widetilde{W}$ is a $\bbP^{0,H}$-Brownian motion. Since $\partial_y p^0 / p^0$ is locally Lipschitz, the previous SDE has a unique solution in law, therefore $\prob^{0,H}$ must be the law of $Y^{0,H}$.
\end{proof}

We are now ready to prove the convergence results.

\begin{proof}[Proof of Theorem \ref{thm: convergence}]
 The existence of Glosten-Milgrom equilibria follows from \eqref{eq: dist tv delta} and Theorem \ref{thm: existence equilibrium} directly. We will prove the statements on convergence in what follows.
\begin{enumerate}
\item[i)]   First note that $\lim_{\delta \downarrow 0} p^{\delta}(y,t) =
 p^0(y,t)$ follows from the argument which leads to \eqref{eq: h delta
   conv}. Moreover, this immediately implies the convergence of bid and ask prices
 as given in  i) since  $a^\delta(y, t) = p^\delta(y+\delta, t)$ and
 $b^{\delta}(y, t) = p^{\delta}(y-\delta, t)$. To verify the
 convergence of the market depth, observe that
 \[
 \begin{split}
 a^\delta(y,t) - p^\delta(y,t) &= \prob^\delta_{y+\delta}[Y^\delta_{1-t} \geq y_\delta] - \prob^\delta_y[Y^\delta_{1-t}\geq y_\delta] = \prob^{\delta}_0[Y^\delta_{1-t} = y_\delta - y-\delta]\\
 &= \overline{\prob}\bra{\overline{Y}_{1-t} = \frac{y_\delta - y -\delta}{\delta}},
 \end{split}
 \]
 where $\overline{Y}_{1-t}$ is the difference of two independent
 Poisson random variables with the common parameter $(1-t)\beta =
 (1-t)(2\delta^2)^{-1}$ under $\overline\prob$. Recall that the
 difference of two independent Poissons has the so-called
 \emph{Skellam} distribution (see \cite{Skellam}). Thus,
 $\overline{\prob}(\overline{Y}_{1-t}= k) = e^{-2\mu} I_{|k|}(2\mu)$,
 where $I_{|k|}(\cdot)$ is the \emph{modified Bessel function of the
   second kind} and $\mu=(1-t)(2\delta^2)^{-1}$. As a result
 \[
 \begin{split}
  \frac{1}{\delta} (a^\delta(y, t) - p^\delta(y, t)) &= \frac{1}{\delta}\overline{\prob}\bra{\overline{Y}_{1-t} = \frac{y_\delta - y -\delta}{\delta}}\\
  &= \frac{1}{\delta} \exp\pare{-\frac{1-t}{\delta^2}} I_{\left|\frac{y_\delta-y-\delta}{\delta}\right|}\pare{\frac{1-t}{\delta^2}}\\
  &\rightarrow \frac{1}{\sqrt{2\pi (1-t)}} \exp\pare{-\frac{(y_0-y)^2}{2(1-t)}}, \quad \text{ as } \delta \downarrow 0.
 \end{split}
 \]
 Here the convergence follows from \eqref{eq: y delta conv} and
 \cite[Theorem 2]{Athreya}, which states that the density of the Skellam
 distribution converges to the density of the normal after appropriate
 rescaling. Similar argument shows that $(p^\delta(y,t)-
 b^\delta(y,t))/\delta$ converges to the same function. This
 establishes the convergence of market depths given in i) since $\partial_y p^0(y,t)$ is exactly the normal density above.

\item[ii)]  Recall from Section \ref{sec: bridge construction} and the
  discussion preceding Lemma \ref{lem: cumulative order conv} that,
  for each $\delta>0$, the distribution of the cumulative order
  process in $\overline{\F}^{\delta}$ on the set $[\tv=1]$
  (resp. $[\tv=0]$) is the same as the distribution of $Y^\delta$
  conditioned on $Y^\delta_1 \geq y_\delta$ (resp. $Y^\delta_1 <
  y_\delta$). However, Lemma \ref{lem: cumulative order conv} has already shown that $Law(Y^\delta \,|\, Y^\delta_1 \geq y_\delta) \Rightarrow Law(Y^{0,H})$, where $Y^{0,H} = B^0 + W$. Since $Law(Z^\delta) \Rightarrow Law(W)$ as $\delta \downarrow 0$, it follows from \cite[Proposition VI.1.23]{Jacod-Shiryaev} that $Law(X^{B, \delta}) \Rightarrow Law(B^0)$ as $\delta \downarrow 0$. The convergence of $Law(X^{S, \delta})$ can be similarly proved.

\item[iii)]  This now follows from  Remark \ref{r:kbe}.
\end{enumerate}
\end{proof}

\bibliographystyle{siam}
\bibliography{biblio}
\end{document}